\documentclass[11pt, leqno]{amsart}
\usepackage{amsfonts,amssymb,amscd,amsmath,amsthm}
\usepackage{tikz}
\usepackage[nice]{nicefrac}
\usepackage{esint}
\usepackage{lmodern}
\usepackage{leftindex}
\usepackage{mathrsfs}
\usepackage{qsymbols}
\usepackage{mathtools}
\mathtoolsset{showonlyrefs}
\usepackage{dsfont}
\usepackage{graphicx}
\usepackage{latexsym}
\usepackage{chngcntr}
\usepackage{paralist}
\usepackage[parfill]{parskip}
\usepackage{bm}
\usepackage{enumitem}
\usepackage{tikz-cd}
\usetikzlibrary{arrows}
\usepackage{csquotes}
\usepackage{todonotes}
\usepackage[plainpages=false,pdfpagelabels,
backref=page,
citecolor=red]{hyperref}
\usepackage{xcolor}
\hypersetup{
	colorlinks,
	linkcolor={cyan!90!black},
	citecolor={magenta},
	urlcolor={green!40!black}}

\let\oldproofname=\proofname
\renewcommand{\proofname}{\rm\bf{\oldproofname}}
\renewcommand{\emph}[1]{{\rm\bf{#1}}}
\setlist[enumerate]{label=\textup{(\roman*)}}
\newtheorem{theorem}{Theorem}[section]
\newtheorem{lemma}[theorem]{Lemma}
\newtheorem{proposition}[theorem]{Proposition}
\newtheorem{corollary}[theorem]{Corollary}
\newtheorem{assumption}[theorem]{Assumption}
\theoremstyle{definition}
\newtheorem{definition}[theorem]{Definition}
\newtheorem{remark}[theorem]{Remark}
\numberwithin{equation}{section}
\newcommand{\loc}{\mathrm{loc}}
\DeclareMathOperator{\e}{e}
\DeclareMathOperator{\Div}{div}
\DeclareMathOperator{\sgn}{sgn}

\DeclareMathOperator{\cc}{c}

\newcommand{\1}{\boldsymbol{1}}
\newcommand{\C}{\mathbb{C}}
\newcommand{\N}{\mathbb{N}}

\newcommand{\R}{\mathbb{R}}

\newcommand{\smooth}[1][]{\rC_{\cc}^{\infty}(#1)}
\newcommand{\W}{\mathrm{W}}
\DeclareRobustCommand{\Wdot}{\dot{\W}\protect{\vphantom{W}}}
\renewcommand{\L}{\mathrm{L}}
\renewcommand{\d}{\mathrm{d}}

\DeclareMathOperator{\Avg}{Avg}
\newcommand{\les}{\lesssim}
\renewcommand{\S}{\mathrm{S}}
\let\ii\i
\renewcommand{\i}{\mathrm{i}}

\DeclareMathOperator{\re}{Re}
\newcommand{\rC}{\mathrm{C}}
\newcommand{\cD}{\mathcal{D}}
\newcommand{\cF}{\mathcal{F}}
\newcommand{\cL}{\mathcal{L}}
\newcommand{\cM}{\mathcal{M}}
\newcommand{\cP}{\mathcal{P}}
\newcommand{\cR}{\mathcal{R}}
\newcommand{\cS}{\mathcal{S}}
\newcommand{\cT}{\mathcal{T}}
\newcommand{\dom}{\mathrm{dom}}

\newcommand{\sub}{\subseteq}

\newcommand{\Meyers}{m_+}
\def\Xint#1{\mathchoice
	{\XXint\displaystyle\textstyle{#1}}%
	{\XXint\textstyle\scriptstyle{#1}}%
	{\XXint\scriptstyle\scriptscriptstyle{#1}}%
	{\XXint\scriptscriptstyle%
		\scriptscriptstyle{#1}}%
	\!\int}
\def\XXint#1#2#3{{\setbox0=\hbox{$#1{#2#3}{%
				\int}$ }
		\vcenter{\hbox{$#2#3$ }}\kern-.6\wd0}}
\def\barint{\,\Xint-} 
\title[Meyers exponent rules the first-order approach]{Meyers exponent rules the first-order approach to second-order elliptic boundary value problems}
\author{Pascal Auscher}
\email{pascal.auscher@universite-paris-saclay.fr}
\address{Universit\'e Paris-Saclay, CNRS, Laboratoire de Math\'ematiques d'Orsay, 91405 Orsay}
\author{Tim B\"ohnlein}
\author{Moritz Egert}
\email{boehnlein@mathematik.tu-darmstadt.de, egert@mathematik.tu-darmstadt.de}
\address{Fachbereich Mathematik, Technische Universit\"at Darmstadt, Schlossgartenstr.~7, 64289 Darmstadt, Germany}
\subjclass[2020]{Primary: 35J25, 47F10. Secondary: 35B65, 46E35.}
\date{March 31, 2025}
\dedicatory{}
\keywords{Elliptic boundary value problems, a priori estimates, Meyers exponent, reverse H\"older estimates, Hodge projector, Cauchy--Riemann systems, operator-valued Fourier multipliers}
\begin{document}
	\begin{abstract}
        The first-order approach to boundary value problems for second-order elliptic equations in divergence form with transversally independent complex coefficients in the upper half-space rewrites the equation algebraically as a first-order system, much like how harmonic functions in the plane relate to the Cauchy–Riemann system in complex analysis. It hinges 
        on global $\L^p$-bounds for some $p>2$ for the resolvent of a perturbed Dirac-type operator acting on the boundary.  
        At the same time, gradients of local weak solutions to such equations exhibit  higher integrability for some $p>2$, expressed in terms of weak reverse H\"older estimates.  We show that the optimal exponents for both properties coincide.
        Our proof relies on a simple but seemingly overlooked connection with operator-valued Fourier multipliers in the tangential direction.
    \end{abstract}
	
	\maketitle

	\section{Introduction}
	We work in Euclidean space $\R^{1+n} \coloneqq \{(t,x): t \in \R, x \in \R^n\}$ and consider pure         second-order operators in divergence form 
        \begin{align}
        \label{eq: the guy}
            \cL U \coloneqq - \Div_{t, x} (A \nabla_{t,x} U)
        \end{align}
        with $t$-independent coefficients $A=A(x)$ that are bounded, complex-valued, measurable and strongly elliptic (see Section~\ref{Sec: L2 theory} for precise definitions). Historically, these operators arise as pullbacks of the Laplacian from the domain above a Lipschitz graph. Since the pioneering work of Dahlberg~\cite{Dahlberg}, harmonic analysts have shown a particular interest in a priori estimates, uniqueness, and solvability of boundary value problems for the equation $\cL U = 0$ in the upper half-space $\R^{1+n}_+$ with prescribed data of Dirichlet- and Neumann-type in $\L^p$-spaces. 
        
        When the coefficients of $\cL$ are real, such questions can efficiently be studied through the associated elliptic measure and layer potentials~\cite{HKMP, HKMP2}. These tools are not available for complex coefficients since they rely on the maximum principle and pointwise estimates of weak solutions in style of the celebrated DeGiorgi--Nash--Moser theory. 

        An alternative measure of regularity for weak solutions, independent of whether the coefficients are real or complex, is the self-improving property of weak reverse H\"older estimates for {the gradient of} local $\cL$-harmonic functions, first established by Meyers \cite{Meyers_Reverse-Holder}. This property states that there exists $p>2$ and a constant $C \geq0$ (both depending only on ellipticity and dimension) such that for every axes-parallel cube $Q \subseteq \R^{1+n}$ and every weak solution to $\cL U = 0$ in $2Q$ we have the weak reverse H\"older estimate
        \begin{align*}
            \bigg(\barint_Q |\nabla_{t,x} U|^p \, \d (t, x) \bigg)^{\frac{1}{p}} \leq C \bigg(\barint_{2Q} |\nabla_{t,x} U|^2 \, \d (t, x) \bigg)^{\frac{1}{2}}.
        \end{align*} 
        We denote the supremum over all such $p > 2$ by $m_+(\cL)$ and refer to it as the \emph{Meyers exponent} of $\cL$. 
        The original proof relies on real-variable methods, whose underlying ideas originate from Gehring’s self-improvement property for conformal mappings~\cite{Gehring_Quasiconformal_mapping}. Notably, it implies that weak solutions are H\"older continuous when $n = 1$. However, no direct link between the Meyers exponent and a priori estimates or solvability of boundary value problems is known so far.

        {In order to treat boundary value problems for operators with complex coefficients, the first author, together with Ros\'en (Axelsson) and McIntosh, introduced a ‘first-order approach’ in \cite{A-A-McI_L2_BVP}}, which rewrites the equation $\cL U = 0$ algebraically as a first-order system $\partial_t F + DB F = 0$, much like how harmonic functions in the plane relate to the Cauchy–Riemann system in complex analysis. Here, $D$ is a Dirac operator and the perturbation $B$ arises from the coefficients $A$ through an explicit, yet intricate transformation on elliptic matrices. This approach applies to all equations as above and leads to remarkable results that all, in one way or the other, impose an $\L^p$-type topology on the data and the solution $F$ when $p$ belongs to a specific range of exponents $[2, p_+(DB))$. Let us mention the explicit construction of solutions through a $DB$-semigroup~\cite{Amenta-Auscher_Thesis, Auscher-Stahlhut_Diss}, representation and uniqueness of solutions to boundary value problems~\cite{Auscher-Egert, A-M_Rep-and-Uniqu-via-FO}, bounds for layer potentials beyond singular integral operators~\cite{Auscher-Egert, Rosen}, well-posedness of Neumann problems for equations in block form~\cite{Auscher-Egert_book, A-M_Rep-and-Uniqu-via-FO}, and identification of Hardy spaces~\cite{Amenta-Auscher_Thesis, Auscher-Egert_book, Auscher-Stahlhut_Diss}. The limiting exponent $p_+(DB)$ was introduced in \cite{Auscher-Stahlhut_Diss} as the upper endpoint of the maximal interval of exponents $p$ around $2$ for which $DB$ is bisectorial in $\L^p(\R^n)$, that is, $DB$ satisfies a specific type of resolvent estimates along the imaginary axis.

        A natural and frequently asked question is how this exponent $p_+(DB)$ relates to the underlying second-order equation. {Our main result offers a concise answer.

        \begin{theorem}
        \label{Meyers meets DB: Main Thm: Easy formulation}
        The exponent $p_+(DB)$ from the first-order approach and the Meyers exponent $\Meyers(\cL)$ coincide.
                    \end{theorem}

        We highlight again that the two exponents have fundamentally different natures.} The Meyers exponent $m_+(\cL)$ primarily captures local properties of weak solutions in the interior, uniformly across all scales, whereas $p_+(DB)$ relates to a {global estimate at the boundary.}

        Theorem~\ref{Meyers meets DB: Main Thm: Easy formulation} confirms that results obtained via self-improvement properties for second-order equations and extrapolation in the first-order approach are compatible in terms of admissible exponents. In particular, this shows that $p_+(DB)$ is not merely an artifact of the first-order approach, which constructs specific global solutions; rather, it genuinely reflects interior behavior of \emph{all} local weak solutions in the second-order setting. This may be surprising.

        \subsection{Principal ideas for the proof of Theorem~\ref{Meyers meets DB: Main Thm: Easy formulation}}
        
        We successively replace the weak reverse H\"older estimates with other $\L^p$-estimates that hold for the same range of exponents $p$, with the possible exception of the endpoints. 
        
        The \emph{first idea} is to `globalize' by replacing the set of reverse H\"older estimates (one for each cube) with a global estimate for one single operator. 
        
        Multiple choices for the global operator are possible. A particularly easy proof can be given for the Hodge projector $\nabla_{t,x} \cL^{-1} \Div_{t,x}$ by using an extrapolation technique due to Shen~\cite{Shen_Lp-extrapolation}. The general theory of $\L^p$-families associated with divergence-form operators from the textbook~\cite{Auscher-Egert_book} allows us to replace the Hodge projector further with $s\nabla_{t,x}(1+s^2 \cL)^{-N}$ for some large $N$, as long as the $\L^p$-bound remains uniform with respect to $s >0$. These higher-order resolvents offer decay properties that are needed for technical reasons only.

        The \emph{second idea} is to use a seemingly overlooked connection with operator-valued Fourier multipliers through the partial Fourier transform in the $t$-variable already at the level of second-order theory, hence not involving the first-order approach. 
        
        Indeed, partial Fourier transform turns the homogeneous gradient $\nabla_{t,x}$ in $\R^{1+n}$ into a family $S_\tau \coloneqq [\i \tau, \nabla_x]^\top$ of $\tau$-dependent inhomogeneous gradients in $\R^n$. Since the coefficients of $\cL$ are independent of $t$, this operator corresponds to the family $L_\tau = S_\tau^* A S_\tau$ of inhomogeneous second-order operators. Hence, the global operator from above can be written as an operator-valued Fourier multiplier
        \begin{align*}
            s \nabla_{t,x}(1+s^2 \cL)^{-N} = \cF_t^{-1} \big(s S_\tau (1+s^2 L_\tau)^{-N}\big) \cF_t.
        \end{align*}
        The operators $L_\tau$ previously appeared in a second-order reformulation of $\L^p$-bisectori\-ality of $DB$ in \cite{Auscher-Stahlhut_Diss}, and it turns out that the pieces fit together beautifully: The condition on $L_\tau$ that arises from $\L^p$-bisectoriality of $DB$ is exactly the one that is needed to verify that the Fourier symbol of $s \nabla_{t,x}(1+s^2 \cL)^{-N}$ is an $\cR$-Mihlin symbol with values in the bounded operators on $\L^p(\R^n)$. Thus, this operator is bounded on $\L^p(\R; \L^p(\R^n)) \cong \L^p(\R^{1+n})$ by virtue of Weis' operator-valued multiplier theorem~\cite{Weis_Lp-multipliers}.
        
        The above line of reasoning shows that $p_+(DB)$ is at least as large as $\Meyers(\cL)$. For the converse, we reverse the 'globalizing' step in one dimension lower and replace the estimates from \(\L^p\)-bisectoriality of \(DB\) with weak reverse H\"older estimates for the gradient of solutions to \(L_\tau u = 0\). Care is required, however, as we need to develop the theory uniformly in the Fourier variable \(\tau\). It then remains to show that the Meyers exponents for the operators \(L_\tau\) cannot exceed that of \(\cL\). This follows immediately from the observation that if \(u\) is a weak solution to \(L_\tau u = 0\), then \(U(t, x) = \e^{\i \tau t} u(x) \) defines a solution to \(\cL U = 0\) in one dimension higher.

        \subsection{Structure of the paper and additional findings}
        
        Figure~\ref{fig: Roadmap} comprises the structure of the proof of our main result. The full equivalence reveals some further characterizations of $p_+(DB)$ that have consequences for the operators $DB$ and $\cL$ beyond the scope of Theorem~\ref{Meyers meets DB: Main Thm: Easy formulation}. They mostly involve further exponents linked to the operators $L_\tau$ which are interesting on their own. We summarize these additional findings in Theorem~\ref{Thm: characterization of I(DB)} at the end of our paper.
        
        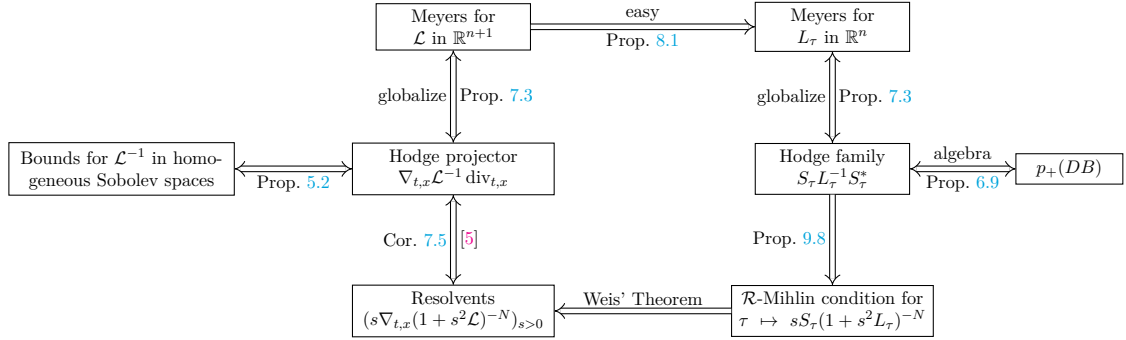
\begin{figure}[h]
            \centering
            \begin{tikzpicture}
            \begin{scope}[scale=0.63, transform shape] 
                %
                %
                \node (Meyers n+1) at (0,12) [text width=3cm, align=center,draw]{Meyers for $\cL$ in $\R^{n+1}$};
                \node (Meyers n) at (8,12) [text width=3cm, align=center,draw]{Meyers for $L_\tau$ in $\R^{n}$};
                \node (Hodge n+1) at (0,9) [text width=4cm, align=center,draw]{Hodge projector $\nabla_{t,x} \cL^{-1} \Div_{t,x}$};
                \node (Hodge n) at (8,9) [text width=3cm, align=center,draw]{Hodge family $S_\tau L_\tau^{-1} S_\tau^*$};
                                \node (Resolvents n+1 regular) at (0,6) [text width=4cm, align=center, draw]{Resolvents $(s \nabla_{t,x} (1+s^2 \cL)^{-N})_{s >0}$};
                \node (Multiplier) at (8,6) [text width=4cm, align=center, draw]{$\cR$-Mihlin condition for  $\tau~\mapsto~s S_\tau (1+s^2 L_\tau)^{-N}$};
                \node (DB) at (13, 9) [text width=2cm, align=center, draw]{$p_+(DB)$};
                \node (Isom) at (-7,9)  [text width=4.5cm, align=center, draw]{Bounds for $\cL^{-1}$ in homogeneous Sobolev spaces};
                %
                %
                %
                %
                \draw[-implies,double equal sign distance] (Meyers n+1) -- (Meyers n) node[midway, above]{easy} node[midway, below]{Prop.~\ref{M(cL) < M(L_t): Prop: M(cL) < M(L_t)}};
                \draw[implies-implies,double equal sign distance] (Meyers n+1) -- (Hodge n+1) node[midway,left]{globalize} node[midway, right]{Prop.~\ref{M(L_t) = P_+(L_t): Prop: M(L_t) = P_+(L_t)}};
                \draw[implies-implies,double equal sign distance] (Hodge n+1) -- (Resolvents n+1 regular) node[midway, right]{\cite{Auscher-Egert_book}} node[midway, left]{Cor.~\ref{M(L_t) = P_+(L_t): Cor: M(cL) = Stuff from AE}};                
                \draw[-implies,double equal sign distance] (Multiplier) -- (Resolvents n+1 regular)  node[midway, above] {Weis' Theorem} ;
                \draw[implies-implies,double equal sign distance] (Meyers n) -- (Hodge n) node[midway,left]{globalize} node[midway, right]{Prop.~\ref{M(L_t) = P_+(L_t): Prop: M(L_t) = P_+(L_t)}};
                \draw[-implies,double equal sign distance] (Hodge n) -- (Multiplier) node[midway, left] {Prop.~\ref{M(L_t) < M(cL): Prop: SFE for m_s(tau)}};
                \draw[implies-implies,double equal sign distance] (Hodge n) -- (DB) node[midway, above] {algebra} node[midway, below] {Prop.~\ref{DB: Prop: I(DB) = Hodge((L_tau)_tau)}};
                \draw[implies-implies,double equal sign distance] (Isom) -- (Hodge n+1) node[midway, below] {Prop.~\ref{Inhom Hodge theory: Prop: Hodge-range via Hodge}};
            \end{scope}
            \end{tikzpicture}
            \caption{Roadmap to Theorem~\ref{Meyers meets DB: Main Thm: Easy formulation}. The exponent $p_+(DB)$ coincides with the Meyers exponent for $\cL$.}
            \label{fig: Roadmap}
        \end{figure}
    
        Sections~\ref{Sec: Notation}–\ref{Sec: DB} contain preliminary material. The proof of Theorem~\ref{Meyers meets DB: Main Thm: Easy formulation} is presented in Sections~\ref{Sec: Meyers(L_tau) = P_+(L_tau)}–\ref{Sec: Proof main n > 2}. For reasons of homogeneity related to Sobolev embeddings, our argument applies in boundary dimensions $n\geq2$. In the case $n=1$, it is well-known from~\cite{Auscher-Stahlhut_Diss} that $p_+(DB) = \infty$ and from \cite[App.~B]{Auscher-Tchamitchian_book} that $m_+(\cL) = \infty$. We provide a direct argument for $\Meyers(\cL) = \infty$ in Section~\ref{Sec: M(L) = infinity if n = 2}.

        \subsection{Acknowledgment}

        The second and third authors acknowledge the support of the CNRS and the Laboratoire de Mathématiques d'Orsay, where this project was {partly} carried out during two research stays in September 2023 and October 2024. The authors are also grateful to Sebastian Bechtel for highlighting the square function techniques in~\cite{Bechtel-Ouhabaz_ODEs, Kunstmann-Weis_book}. A CC-BY 4.0 
\url{https://creativecommons.org/licenses/by/4.0/} public copyright license has been applied by the authors to the present document and will be applied to all subsequent
versions up to the Author Accepted Manuscript arising from this
submission.

	
	\section{Notation} \label{Sec: Notation}
	
	Most of our notation is standard and we make use of the following additional conventions:
	
	\begin{itemize}
	
		\item We abbreviate $\R^* \coloneqq \R \setminus \{ 0 \}$.
		
		\item We write $Q(x, r)$ for the open, axes-parallel cube with center $x$ and length $2r$. Likewise, $B(x,r)$ denotes the ball centered at $x$ with radius $r$.
		
		\item For suitable exponents $p$, we define conjugate indices $p' \coloneqq \nicefrac{p}{(p-1)}$ (H\"older), $p^* \coloneqq \nicefrac{p n}{(n - p)}$ (upper Sobolev) and $p_* \coloneqq \nicefrac{pn}{(n+p)}$ (lower Sobolev), where the ambient dimension $n$ will be clear from the context.
		
		\item We write $X^*$ for the Banach space of bounded and anti-linear functionals from $X$ to $\C$. 
		
		\item Our one-dimensional Fourier transform is
		\begin{equation*}
			(\cF_t f)(\tau) \coloneqq \frac{1}{\sqrt{2 \pi}} \int_{\R} f(t) \e^{- \i t \cdot \tau} \, \d t. 
		\end{equation*}
		
		\item For dilations of functions we write $(\delta_t f) (x) \coloneqq f(tx)$.
		
		\item We use the notation $(f)_E\coloneqq \fint_E f \, \d x \coloneqq \frac{1}{|E|} \int_E f \, \d x$ for averages. 

            \item We use $\lesssim$ and $\gtrsim$ to denote inequalities that hold up to multiplicative constants independent of the relevant quantities.
	\end{itemize}
	
	
	\section{Function spaces and elliptic operators} \label{Sec: L2 theory}
	
	For purely second-order operators, we typically use homogeneous Sobolev spaces.
	
	\begin{definition}
		For $p \in (1, \infty)$, we define $\Wdot^{1,p}(\R^n)$ as the space of all $\L^p_{\loc}$-functions $u$ for which $\nabla_x u$ belongs to $\L^p(\R^n)^n$ modulo $\C$. We endow this space with the norm 
		\begin{equation*}
			\| u \|_{\Wdot^{1,p}} \coloneqq \| \nabla_x u \|_p.
		\end{equation*}
		We also write $\Wdot^{-1,p}(\R^n) \coloneqq \Wdot^{1,p'}(\R^n)^*$.
	\end{definition}

        The same notation will be used in dimension $n+1$, where we write $\nabla_{t,x}$ for the gradient. Throughout this article, we work under the following assumption on the coefficients.

        \begin{assumption}[Ellipticity]
            The coefficients $A: \R^n \to \L^\infty(\R^n; \C^{(1+n) \times (1+n)})$ are strongly elliptic in the sense that there exists $\lambda>0$ such that 
            \begin{align*}
                \re \langle A(x) \xi, \xi \rangle \geq \lambda |\xi|^2 \qquad (x \in \R^n, \xi \in \C^{1+n}).
            \end{align*}
        \end{assumption}

        On the (partial) Fourier side, we work with inhomogeneous operators where $\nabla_{t,x}$ is replaced by inhomogeneous gradients that scale with the Fourier variable. We introduce the natural Sobolev spaces that capture the scaling.
	
	\begin{definition}
		Let $\tau \in \R$ and $p \in (1,  \infty)$.
        
            \begin{enumerate}
                \item We write $S_\tau \coloneqq [\i \tau, \nabla_x]^\top$ and $S_\tau^* \coloneqq [-\i \tau, -\Div_x]$.
                \item The space $\W^{1, p}_\tau(\R^n)$ is the usual inhomogeneous Sobolev space but with the $\tau$-adapted norm
		\begin{equation*}
			\| u \|_{\W^{1,p}_{\tau}} \coloneqq \|S_\tau u\|_p
		\end{equation*}
		and $\W^{-1,p}_{\tau}(\R^n) \coloneqq \W^{1,p'}_{\tau}(\R^n)^*$ is its dual.
            \end{enumerate}
	\end{definition}
	
	\begin{remark}
		When $\tau = 0$, the space $\W^{1, p}_\tau(\R^n)$ is understood as $\Wdot^{1, p}(\R^n)$ and $S_\tau$ as the usual gradient $\nabla_x$.
	\end{remark}

	Let $\tau \in \R$. In view of the Lax--Milgram lemma, we can define the isomorphism
	\begin{align*}
		L_\tau &\coloneqq S_\tau^* A S_\tau \colon \W^{1, 2}_\tau(\R^n) \to \W^{-1,2}_\tau(\R^n),\\
		\langle L_\tau u, v \rangle &\coloneqq \int_{\R^n} A S_\tau u \cdot \overline{S_\tau v} \, \d x.
	\end{align*}
        The ellipticity assumption ensures that the maximal restriction of $L_\tau$ to $\L^2(\R^n)$ is m-accretive and its negative generates a holomorphic $\rC_0$-semigroup of contractions in $\L^2(\R^n)$. Background material can be found e.g.\ in \cite{Kato_book}. For simplicity, we do not distinguish notationally between $L_\tau$ and this negative generator. 
        
        Writing the matrices $A(x)$ according to the $(t,x)$-notation as 
        \begin{align}
        \label{eq: Splitting of A}
		A \coloneqq \begin{bmatrix}
			a & b \\ c & d
		\end{bmatrix}, 
	\end{align}
        we see that $L_\tau$ is the weak interpretation of the divergence-form operator with lower-order terms
        \begin{align}
        \label{eq: Formula L-tau}
           L_\tau u = -\Div_x (d \nabla_x u) - \i \tau \Div_x (cu) - \i \tau b \nabla_x u + \tau^2 a u.
        \end{align}
        
        The main elliptic operator $\cL$ in \eqref{eq: the guy} has an identical weak interpretation with $\tau  =0$ in one dimension higher. In particular, it is also m-accretive.
	
	
	\section{Tools from $\L^p - \L^q$ off-diagonal theory} \label{Sec: Lp-Lq estimates}
	
	We review some abstract theory on off-diagonal estimates for uniformly bounded families $\cT = (T(t))_{t \in I} \subseteq \cL(\L^2)$ defined on some set $I \subseteq \R^*$. For us, $\cT$ will be a semigroup or resolvent family associated with a differential operator or a singleton. Most of the material is taken from~\cite{Auscher-Egert_book}. 
	
	\begin{definition} \label{Lp-Lq estimates: Def: Lp-Lq estimates}
		Let $1 \leq p \leq q \leq \infty$. We say that $\cT$
		\begin{itemize}
			\item is \emph{$\boldsymbol{\L^p - \L^q}$-bounded} if there is $C > 0$ such that 
			\begin{equation*}
				\| T(t) f \|_q \leq C |t|^{\frac{n}{q} - \frac{n}{p}} \| f \|_p 
			\end{equation*}
			for all $t \in I$ and $f \in \L^p \cap \L^2$. 
			
			\item satisfies \emph{$\boldsymbol{\L^p - \L^q}$ off-diagonal estimates} if there are $C, c > 0$ such that 
			\begin{equation*}
				\| \1_F T(t) \1_E f \|_q \leq C |t|^{\frac{n}{q} - \frac{n}{p}} \e^{-c \frac{\d(E, F)}{|t|}}\| \1_E f \|_p 
			\end{equation*}
			for all $t \in I$, measurable sets $E, F \subseteq \R^n$ and $f \in \L^p \cap \L^2$. 
		\end{itemize}
		For $p = q$, we speak of \emph{$\boldsymbol{\L^p}$-boundedness} and \emph{$\boldsymbol{\L^p}$ off-diagonal estimates}, respectively.  
	\end{definition}
	
	These notions interpolate as expected and we cite the general principle in the form that we need later on. 
	
	\begin{lemma} \label{Lp-Lq estimates: Lem: Lp -> Lq + L2 ODE implies Lr -> Ls ODE}
		Let $1 \leq p < r \leq \sigma < q \leq \infty$. Suppose that $\cT$ satisfies $\L^2$ off-diagonal estimates and that it is $\L^p - \L^q$-bounded. Then $\cT$
        \begin{enumerate}
            \item satisfies $\L^r - \L^\sigma$ off-diagonal estimates and
            \item is $\L^r$-bounded.
        \end{enumerate}
	\end{lemma}

        \begin{proof}
        Upon modifying $p$ and $q$ slightly, but preserving the relation with $r$ and $\sigma$, we can assume that $\cT$ satisfies $\L^p - \L^q$ off-diagonal estimates, see the interpolation principle in \cite[Lem.~4.14]{Auscher-Egert_book}. This implies $\L^p$ and $\L^q$ off-diagonal estimates~\cite[Rem.~4.8]{Auscher-Egert_book} and (i) follows again by interpolation. Finally, (i) implies (ii) by taking $\sigma = r$ and $E=F=\R^n$ in the definition of off-diagonal estimates.
        \end{proof}

        There is also a useful bootstrapping mechanism~\cite[Lem.~4.4]{Auscher-Egert_book}. (The reference uses $q=2$ as reference point but this does not have an impact on the argument.)

        \begin{lemma} \label{Lp-Lq estimates: Triangle Argument}
		Let $1 \leq p < r \leq \sigma < q \leq \infty$ and suppose that $\cT$ is $\L^q$-, $\L^p$- and $\L^\sigma-\L^q$-bounded. Then there exists some $k \in \N$ such that $\cT^k = (T(t)^k)_{t \in I}$ is $\L^r - \L^q$-bounded.
	\end{lemma}

        We remark that implicit constants in both results depend only on $p$, $q$, $r$, $\sigma$ and the constants in the assumption~\cite{Auscher-Egert_book}.
	
	
	\section{$\tau$-dependent Hodge projectors} \label{Sec: Inhom Hodge theory}
	
	We begin with the theory outlined in Figure~\ref{fig: Roadmap} and introduce a family of $\tau$-dependent Hodge projectors. Again, this comprises the corresponding theory for $\cL$ because this operator is of the same type as $L_0$ in one dimension higher.

        \begin{definition}
            Let $\tau \in \R$. The operator
	\begin{equation*}
		S_{\tau} L_{\tau}^{-1} S_{\tau}^* = \begin{bmatrix}
			\tau^2 L_{\tau}^{-1} & -\i \tau L_{\tau}^{-1} \Div_x \\
			-\i \tau \nabla_x L_{\tau}^{-1} & -\nabla_x L_{\tau}^{-1} \Div_x
		\end{bmatrix}
	\end{equation*}
            is called \emph{Hodge projector} associated with $L_\tau$.
        \end{definition}

        By the Lax--Milgram lemma, we have for all $u \in \W^{1,2}_\tau$ and all $\tau \in \R$ that
        \begin{align*}
            \| L_{\tau}^{-1} u \|_{\W^{1, 2}_{\tau}} \leq \lambda^{-1} \| u \|_{\W^{-1, 2}_{\tau}}.
        \end{align*}
        Since $S_\tau: \W^{1, 2}_{\tau} \to \L^2$ is isometric, the Hodge projector is bounded on $\L^2$ with norm $\lambda^{-1}$ independently of $\tau$. Its $\L^p$-boundedness can be characterized as follows.

        \begin{proposition} \label{Inhom Hodge theory: Prop: Hodge-range via Hodge}
		Let $\tau \in \R$ and $p \in (1,\infty)$. The following are equivalent:
                \begin{enumerate}
                \item $S_\tau L_\tau^{-1} S_\tau^*$ is $\L^p$-bounded.
                \item There is a constant $C > 0$ such that 
		          \begin{equation*}
		          	\| L_{\tau}^{-1} u \|_{\W^{1, p}_{\tau}} \leq C \| u \|_{\W^{-1, p}_{\tau}} \qquad (u \in \W^{-1, p}_{\tau} \cap \W^{-1, 2}_{\tau}).
		          \end{equation*}
                \end{enumerate}
            In this case, the bound in (i) and the constant in (ii) can be taken the same up to a factor depending only on $p$ and $n$. 
        \end{proposition}

        \begin{proof}
        The implication (ii) $\implies$ (i) follows as above since $S_\tau: \W^{1, p}_{\tau} \to \L^p$ is isometric. The converse for $\tau=0$ is done in \cite[Lem.~13.4]{Auscher-Egert_book} and the argument makes the dependence of constants transparent. It remains to prove (i) $\implies$ (ii) in the inhomogeneous case $\tau \neq 0$, which, in fact, is much easier than the homogeneous counterpart.

        To this end, we assume (i) with $\L^p$-bound $C'$ and let $u \in \W^{-1, p}_\tau \cap \W^{-1, 2}_\tau$. Since $S_\tau^* S_\tau = (\tau^2 -\Delta_x)$, we can write $u = S_\tau^* v$ with $v \coloneqq S_\tau (\tau^2 - \Delta_x)^{-1} u \in \L^p \cap \L^2$. In this way, we obtain
        \begin{equation*}
        \| L_\tau^{-1} u \|_{\W^{1, p}_\tau} = \| S_\tau L_\tau^{-1} S_\tau^* v \|_p \leq C' \| v \|_p
        \end{equation*}
        and we need to bound $v$ independently of $\tau$. Through dilations with parameters $\tau, \tau^{-1}$, we find
        \begin{equation*}
        v = S_\tau (\tau^2 - \Delta_x)^{-1} u = \tau^{-1} \delta_\tau \Big(S_1 (1- \Delta_x)^{-1} \Big) \delta_{\tau^{-1}} u.
        \end{equation*}
        A composition of three bounded operators acts on the right-hand side: $\tau^{-1} \delta_\tau \colon \L^p \to \L^p$ has norm $|\tau|^{-1-n/p}$ by the transformation rule, $S_1 (1-\Delta_x)^{-1} \colon \W^{-1,p} \to \L^p$ is bounded by the lifting properties of the Bessel potentials~\cite[Sec.~1.3.1]{Grafakos-Modern_book}, and $\delta_{\tau^{-1}} \colon \W^{-1,p}_\tau \to \W^{1,p}$ has norm $|\tau|^{1+n/p}$ as the dual of $|\tau|^n \delta_{\tau}: \W^{1,p'} \to \W^{1,p'}_\tau$. Thus, $\|v\|_p \les \|u\|_{\W^{-1,p}_\tau}$
        with an implicit constant independent of $\tau$, and we are done.
        \end{proof}

	
	\section{The first-order Dirac operator $DB$} \label{Sec: DB}
	
	As a final ingredient for the proof of our main result, we introduce the perturbed Dirac operators $DB$ and the key exponent $p_+(DB)$ as in~\cite{Auscher-Stahlhut_Diss}. There is one new result in this section: we characterize $p_+(DB)$ through the Hodge family $(S_\tau L_\tau^{-1} S_\tau^*)_{\tau \in \R^*}$. This necessitates sorting out certain subtleties related to compatible inverses of linear operators.
    
    Recall the block form of the coefficients of $\cL$ in \eqref{eq: Splitting of A}. The ellipticity assumption on $A$ implies that $A$ is invertible in $\L^{\infty}$ and we can introduce the following matrix-valued functions:
	
	\begin{definition}[{\cite{Auscher-Axelsson_weighted_max_reg}}]
		We let 
		\begin{equation*}
			\underline{A} \coloneqq \begin{bmatrix}
				1 & 0 \\ c & d
			\end{bmatrix}, \quad  
			\overline{A} \coloneqq \begin{bmatrix}
				a & b \\ 0 & 1
			\end{bmatrix}, \quad
            \text{and} \quad B = \underline{A} \overline{A}^{-1}.
		\end{equation*}
	\end{definition}

        Next, we define (perturbed) Dirac operators. 
	
	\begin{definition} \label{DB: Def: D and DB}
		We define the \emph{Dirac operator} $D$ in the distributional sense as
		\begin{equation*}
			D \coloneqq \begin{bmatrix}
				0 & \Div_x \\ - \nabla_x & 0
			\end{bmatrix} \colon \L^1_{\loc}(\R^n)^{1+n} \to \cD'(\R^n)^{1+n}
		\end{equation*}
		and for $p \in (1, \infty)$ we denote by $D_p$ its maximal restriction to a linear operator in $\L^p$. The \emph{perturbed Dirac operator} is the composition operator $DB$ and its part in $\L^p$ is $(DB)_p = D_p B$.
	\end{definition}

        \begin{remark} \label{DB: Rem: Standard properties}
        The operator $D_p$ in $\L^p$ is closed as a first-order differential operator with maximal $\L^p$-domain. Since $B$ is bounded, also $(DB)_p$ is closed.

        \end{remark}

         We usually drop the subscript $p$ from our notation when $p = 2$. It is shown in \cite[Prop.~2.5]{A-K-McI_quadratic_estimates} that $DB$ is a \emph{bisectorial} operator in $\L^2$: There exists some  $\mu \in (0, \nicefrac{\pi}{2})$ such that the spectrum of $DB$ is contained in the closure of the bisector
	\begin{equation*}
		\S_{\mu} \coloneqq \{ z \in \C \colon |\arg(\pm z)| < \mu \}
	\end{equation*}
	and for each $\nu \in (\mu, \nicefrac{\pi}{2})$ there is $C > 0$ such that 
	\begin{equation*}
		\| (1 + z DB)^{-1} \|_{\L^2 \to \L^2} \leq C \qquad (z \in \C \setminus \overline{\S_{\nu}}).
	\end{equation*}
        
        \begin{theorem}[{\cite{Auscher-Stahlhut_DB-bisectoriality}, \cite[Thm.~3.6]{Auscher-Stahlhut_Diss}}]
        There is a maximal open interval $I(DB) \subseteq (1,\infty)$ around $2$ such that $(DB)_p$ is bisectorial in $\L^p$ for all $p \in I(DB)$. Moreover, bisectoriality fails at the endpoints.
        \end{theorem}
        
        This result gives rives to the key exponent $p_+(DB)$.
	
	\begin{definition}[{\cite[Sec.~3.2]{Auscher-Stahlhut_Diss}}] \label{DB: Def: Auscher-Stahlhut interval}
		We denote by $p_+(DB)>2$ the upper endpoint of $I(DB)$.
	\end{definition}

       \begin{remark} \label{DB: Rem: B invertible}
        Due to our strong ellipticity assumption, $A$ is invertible in $\L^{\infty}$. Hence, $B$ is also invertible, and in particular, multiplication by $B$ induces  an isomorphism on any $\L^p$ space. This would not be true under weaker ellipticity conditions as they are typically imposed for elliptic systems and avoids the discussion of $\L^p$-coercivity, compare with \cite[Rem.~3.5]{Auscher-Stahlhut_Diss}.
	\end{remark}

    	The following considerations were tacitly used in \cite{Auscher-Stahlhut_Diss}. We provide the details for the reader's convenience and to clarify the necessity of restricting ourselves to intervals around $p=2$.
    
        \begin{lemma} \label{DB: Lem: Char of I(DB)}
             Let $p \in (1, \infty)$. Then the following assertions are equivalent. 
        
             \begin{enumerate}
                 \item $((1 + \i t DB)^{-1})_{t \in \R^*}$ is $\L^p$-bounded. 
        
                 \item $p \in I(DB)$. 
        
                 \item $p \in I(DB)$ and the resolvents $(1+\i t (DB)_p)^{-1}$ and $(1+\i t (DB)_2)^{-1}$ agree on $\L^p \cap \L^2$ for all $t \in \R^*$.
             \end{enumerate}
         \end{lemma}
        
        \begin{proof}
        The implication \emph{(iii) $\implies$ (i)} is obvious but the rest requires some work.
        
        \emph{(i) $\implies$ (ii).} We first observe that it is enough to conclude that $(DB)_p$ is bisectorial. Indeed, since the assumption (i) holds for $p=2$ and interpolates with respect to $p$, this would also imply $p \in I(DB)$. To this end, we denote by $R_p(t)$ the unique bounded extension of $(1 + \i t DB)^{-1}$ to $\L^p$ and show that $1 + \i t (DB)_p$ is bijective with $(1 + \i t (DB)_p)^{-1} = R_p(t)$. 
         
        To prove that $1 + \i t (DB)_p$ is surjective, we let $f \in \L^p$ and pick a sequence $(f_j)_j \subseteq \L^p \cap \L^2$ with $f_j \to f$ in $\L^p$. Then $u_j \coloneqq (1 + \i t DB)^{-1} f_j \to R_p (t) f$ in $\L^p$. Since $(u_j)_j \subseteq \L^p \cap \dom(DB)$ with $\i t DB u_j = f_j - u_j \in \L^p$, we also have $(u_j)_j \subseteq \dom(D_pB)$ with $(1 + \i t (DB)_p)u_j = f_j \to f$ in $\L^p$. Now, $1 + \i t (DB)_p$ is closed and $R_p(t) f \in \dom(D_pB)$ as well as $(1 + \i t (DB)_p) R_p (t) f = f$ follow. Hence, $R_p(t)$ is a right-inverse for $1 + \i t (DB)_p$. 
        			
        As for injectivity, we pick $u \in \dom(DB)_p$ with $(1 + \i t (DB)_p) u = 0$ and prove $u = 0$. Given $\varphi \in \rC_{\cc}^{\infty}$, we write
        \begin{align*}
            0 &= \langle R_p(t) 0, \varphi \rangle 
            \\&= \langle (1 + \i t DB)^{-1}(u + \i t (DB)_p u), \varphi \rangle. 
        \intertext{Since $(1 + \i t DB)^{-1}$ is $\L^p$-bounded, the dual family $(1 - \i t B^* D)^{-1}$ is $\L^{p'}$-bounded. We set $v \coloneqq (1 - \i t B^* D)^{-1} \varphi$. Hence, $v \in \L^{p'} \cap \dom(D)$ with $- \i t B^* D v = \varphi - v \in \L^{p'}$, so that even $v \in \dom(D_{p'}) \cap \dom(D)$. Through smooth truncation and convolution, we can approximate $v$ by a sequence $(v_j) \subseteq \rC_{\cc}^\infty$ such that $v_j \to v$ and $Dv_j \to Dv$, both in $\L^{p'}$. Consequently, we can continue by}
            &= \langle u + \i t (DB)_p u, v \rangle
            \\&= \lim_{j \to \infty} \langle u + \i t (DB)_p u, v_j \rangle
            \\&= \lim_{j \to \infty} \langle u, (1 - \i t B^* D)v_j \rangle
            \\&=  \langle u, (1 - \i t B^* D)v \rangle
            \\&= \langle u, \varphi \rangle.
        \end{align*}
        Since $\varphi \in \rC_{\cc}^\infty$ was arbitrary, $u = 0$ follows.
        
         \emph{(ii) $\implies$ (iii).} Let $M \subseteq I(DB)$ be the set of exponents $p$ for which we have compatibility with the $\L^2$-resolvent of $DB$ as stated in (iii). $M$ is non-empty because it contains $2$. We claim that it is open and closed in $I(DB)$ and hence equal to $I(DB)$.
         
        Openess follows directly from \u{S}ne\u{\ii}berg's stability theorem \cite{ABES-Non-Local, Sneiberg_Extrapolation}. Next, we take a sequence $(p_j)_j \subseteq M$ with $p_j \to p \in I(DB)$. Again by \u{S}ne\u{\ii}berg's theorem and for large enough $j$ we have
        \begin{align*}
            (1 + \i t (DB)_p)^{-1} f = (1 + \i t (DB)_{p_j})^{-1} f = (1 + \i t DB)^{-1} f
        \end{align*}
        for all $f \in \L^p \cap \L^{p_j} \cap \L^2$. Now, $p \in M$ follows since $\L^p \cap \L^{p_j} \cap \L^2$ is dense in $\L^p \cap \L^2$.
         \end{proof}

        The following algebraic identity links the $DB$-resolvents with the $\tau$-dependent Hodge projectors. Note that the formulation in \cite{Auscher-Stahlhut_Diss} uses the inhomogeneous gradients $[1, \i t \nabla_x]^\top = \i t S_{-t^{-1}}$ in place of $S_t$.

        \begin{lemma}
            \label{DB: Lem: Algebra}
            Let $\tau \in \R^*$ and $f = [f_\perp, f_\parallel]^\top \in \L^2$. Then
		\begin{equation*}
			(1 - \i \tau^{-1} DB)^{-1} f
			= \overline{A} S_\tau L_{\tau}^{-1} S_\tau^* M f + \begin{bmatrix} b f_{\parallel} \\ f_{\parallel} \end{bmatrix}, \quad \text{where }  M \coloneqq \begin{bmatrix}
				1 & - b \\ 0 & - d
			\end{bmatrix}.
		\end{equation*}
        \end{lemma}

        With the algebra in place, we can prove the characterization of $p_+(DB)$ alluded to above.
	
	\begin{proposition} \label{DB: Prop: I(DB) = Hodge((L_tau)_tau)}
		Let $p \in (1, \infty)$. The following are equivalent:
        \begin{enumerate}
            \item $p \in I(DB)$.
            \item $(S_{\tau} L_{\tau}^{-1} S_{\tau}^*)_{\tau \in \R^*}$ is $\L^p$-bounded.
        \end{enumerate}
        In particular, the set of exponents $p$, for which (ii) holds, is an open interval with upper endpoint $p_+(DB)$.
	\end{proposition}
	
	\begin{proof}
        In view of Lemma~\ref{DB: Lem: Algebra} and since $\overline{A}$, $\overline{A}^{-1}$ and $b$ are in $\L^\infty$, the following two uniform bounds are equivalent:
		\begin{align*}
			\| (1 + \i t DB)^{-1} f \|_p &\les \| f \|_p \qquad (t \in \R^*, f \in \L^p \cap \L^2),\\
			\| S_{\tau} L_{\tau}^{-1} S_{\tau}^* M f \|_p &\les \| f \|_p \qquad (\tau \in \R^*, f \in \L^p \cap \L^2).
		\end{align*}
        According to Lemma~\ref{DB: Lem: Char of I(DB)}, the upper estimate is equivalent to (i). Since $A$ is strongly elliptic, so is $d$ in one dimension lower and therefore $M$ is invertible in $\L^\infty$. Thus, the lower estimate is equivalent to the $\L^p$-boundedness of the Hodge projectors $(S_{\tau} L_{\tau}^{-1} S_{\tau}^*)_{\tau \in \R^*}$ as stated in (ii).
	\end{proof}
	
	
	\section{Global characterization of the Meyers exponent} \label{Sec: Meyers(L_tau) = P_+(L_tau)}
	
	We start with the proof of our main theorem. In this section, we carry out the globalization step in Figure~\ref{fig: Roadmap}.
	
	\begin{definition}
		Let $\tau \in \R$ and $Q = Q(x, r) \sub \R^n$. We say that $u$ is \emph{$\boldsymbol{L_{\tau}}$-harmonic in $Q$} if $u \in \W_{\loc}^{1,2}(Q)$ satisfies $L_{\tau} u = 0$ weakly in $Q$, that is  
		\begin{equation*}
			\int_Q A S_\tau u \cdot \overline{S_\tau \varphi} \, \d x = 0 \qquad (\varphi \in \smooth[Q]).
		\end{equation*}
	\end{definition}
	
	\begin{definition}
            \begin{enumerate}
                \item 	The \emph{Meyers exponent} $\Meyers(L_{\tau})$ of a fixed operator $L_\tau$ is the supremum of all $p \in [2, \infty)$ for which the following holds true: There is a constant $C > 0$ such that for all open axes-parallel cubes $Q \sub \R^n$ and $L_{\tau}$-harmonic $u$ in $2Q$ it follows that
		\begin{equation*}
			\left( \fint_Q |S_\tau u|^p \, \d x \right)^{\frac{1}{p}} \leq C \left( \fint_{2 Q} | S_\tau u|^2 \, \d x \right)^{\frac{1}{2}}.
		\end{equation*}

                \item Given $I \subseteq \R$, the Meyers exponent $\Meyers((L_\tau)_{\tau \in I})$ of the collection of operators $(L_\tau)_{\tau \in I}$ is the supremum of all $p \in [2, \infty)$ for which (i) holds for all $\tau \in I$ and a constant that does not depend on $\tau$.
            \end{enumerate}
	\end{definition}
	
	Refining an argument from \cite{Shen_Lp-extrapolation}, we provide the following characterization.
	
	\begin{proposition}   \label{M(L_t) = P_+(L_t): Prop: M(L_t) = P_+(L_t)}
		Let $n \geq 2$. For any $\tau \in \R$ and $I \subseteq \R$ the following hold true:
        
            \begin{enumerate}
                \item We have $\displaystyle \Meyers(L_{\tau}) = \sup \Big\{p \geq 2 : S_{\tau} L_{\tau}^{-1} S_{\tau}^* \text{ is $\L^p$-bounded} \Big\}$.
                \item We have $\displaystyle \Meyers((L_{\tau})_{\tau \in I}) = \sup \Big\{p \geq 2 : (S_{\tau} L_{\tau}^{-1} S_{\tau}^*)_{\tau \in I} \text{ is $\L^p$-bounded} \Big\}$.
                \end{enumerate}
	\end{proposition}
	
	\begin{proof}
		We provide a proof for (i). Our argument will automatically give the additional uniformity of implicit constants that is required in (ii).
		
		\emph{Step 1: Meyers controls Hodge.} Suppose that $p>2$ is such that $S_\tau L_\tau^{-1} S_\tau^*$ is $\L^p$-bounded. Fix $Q = Q(x, r) \sub \R^n$ and an $L_\tau$-harmonic $u$ in $2Q$. We establish a weak reverse H\"older estimate for $S_\tau u$.
        
        Let $\gamma \in (1, 2)$ be a number that will be fixed later on and $\varphi \in \rC_{\cc}^{\infty}$ be such that $\1_Q \leq \varphi \leq \1_{\gamma Q}$ and $\| \nabla_x \varphi \|_{\infty} \les r^{-1}$ with implicit constant depending on $\gamma$ and dimension. Then $v \coloneqq \varphi (u - \kappa) \in \W^{1,2}$ with $\kappa \coloneqq (u)_{\gamma Q}$ satisfies a global equation 
		\begin{equation*}
			L_{\tau} v  = f - \Div_x (F), 
		\end{equation*}
		on $\R^n$, where 
		\begin{align*}
			f &\coloneqq \Big(- d \nabla_x u \cdot \nabla_x \varphi - \i \tau \kappa c \cdot \nabla_x \varphi \Big) - \tau \Big(\i (u - \kappa) (b + c^T) \nabla_x \varphi + \tau \kappa a \varphi\Big) \\
                &\eqqcolon f_1 - \tau f_2, \\
			F &\coloneqq (u - \kappa) d \nabla_x \varphi - \i \tau \kappa \varphi c.
		\end{align*}
		At this point, we need to start tracking the dependence on $\tau$ carefully. For all $g \in \L^p$, we have $\|\tau g\|_{\W_\tau^{-1,p}} \leq \|g\|_p$. The assumptions $n \geq 2$ and $p>2$ imply $p_* > 1$ and consequently, by the standard Sobolev embedding, we have $\|g\|_{\W_\tau^{-1,p}} \les \|g\|_{p_*}$ for all $g \in \L^{p_*}$. Using the properties of $\varphi$ and then the $\L^p$-boundedness of the Hodge projector through its equivalent formulation in Proposition~\ref{Inhom Hodge theory: Prop: Hodge-range via Hodge}, we obtain
		\begin{align*}
                \|S_\tau u\|_{\L^p(Q)} & \leq |\tau \kappa| |Q|^{\frac{1}{p}} + \|v\|_{{\W^{1, p}_\tau}}
			\\ &\les |\tau \kappa| |Q|^{\frac{1}{p}} + \| f- \Div_x F \|_{\W^{-1,p}_{\tau}}
			\\&\les  |\tau \kappa| |Q|^{\frac{1}{p}} + \| f_1 \|_{p_*} + \| f_2 \|_p  + \|  F  \|_p 
			\\&\les |\tau \kappa| |Q|^{\frac{1}{p}} + r^{-1} \| u - \kappa \|_{\L^p(\gamma Q)} + r^{-1} \| \nabla_x u \|_{\L^{p_*}(\gamma Q)}.
            \end{align*}
		For the first term, we use H\"older's inequality 
            \begin{align*}
            |\tau \kappa| \leq |\gamma Q|^{-\frac{1}{p} - \frac{1}{n}} \|\tau u\|_{\L^{p_*}(\gamma Q)},
            \end{align*}
            whereas the second term can be controlled by the third one via the Sobolev--Poincar\'{e} inequality. In total, we have shown that
            \begin{align*}
			\| S_\tau u \|_{\L^p(Q)} \les r^{-1} \|S_\tau u\|_{\L^{p_*}(\gamma Q)},
		\end{align*}
		which, after dividing both sides by $|Q|^{\nicefrac{1}{p}}$, becomes 
		\begin{equation}
            \label{M(L_t) = P_+(L_t): eq1: M(L_t) = P_+(L_t)}
			\bigg( \fint_Q |S_\tau u |^p \, \d x \bigg)^{\frac{1}{p}} \les 	\bigg( \fint_{\gamma Q} |S_\tau u|^{p_*} \, \d x \bigg)^{\frac{1}{p_*}}. 
		\end{equation}
        
		If $p_* \leq 2$, then the right-hand side is bounded by the $\L^2$-average and we are done. Else, we have $2 < p_* < p$ and the previous argument re-applies to the right-hand side of \eqref{M(L_t) = P_+(L_t): eq1: M(L_t) = P_+(L_t)} with $p_*$ in place of $p$, leading to a new bound by the $(p_*)_*$-average of $S_\tau u$ on $\gamma^2 Q$ as long as $\gamma^2 <2$. After a finite number of iterations, say $N$, this procedure yields an $\L^2$-average on the right. We then choose $\gamma$ a priori such that $\gamma^N < 2$.

            \emph{Step 2: Hodge controls Meyers.} Let $2 < p < \Meyers(L_{\tau})$ and let $C$ be an $\L^p$ weak reverse H\"older constant for $L_\tau$. It suffices to prove that $S_{\tau} L_{\tau}^{-1} S_{\tau}^*$ is $\L^q$-bounded for all $q \in (2, p)$.
            
            From Section~\ref{Sec: Inhom Hodge theory} we know that $S_\tau L_\tau^{-1} S_\tau^*$ is $\L^2$-bounded with norm at most $\lambda^{-1}$. Moreover, if
            $Q \sub \R^n$ is an axes-parallel cube and $f = [f_{\perp}, f_{\parallel}]^{\top} \in \rC^{\infty}_{\cc}$ is such that $f|_{4 Q} = 0$, then $u_{\tau} \coloneqq L_{\tau}^{-1} S_{\tau}^* f$ satisfies $L_{\tau} u_{\tau} = S_{\tau}^* f = 0$ in $2 Q$. Hence, we have
		\begin{equation*}
			\left( \fint_Q |S_{\tau} u_{\tau}|^p \, \d x \right)^{\frac{1}{p}} \leq C \left( \fint_{2 Q} |S_{\tau} u_{\tau}|^2 \, \d x \right)^{\frac{1}{2}}
		\end{equation*} 
		by assumption. We have verified the assumptions for Shen's extrapolation theorem~\cite[Thm.~3.1]{Shen_Lp-extrapolation}, which in turn yields that $S_{\tau} L_{\tau}^{-1} S_{\tau}^*$ is $\L^q$-bounded for all $q \in (2, p)$ with a bound depending on $C$, $\lambda$, $n$ and $q$, see also \cite[Thm.~4.1]{Tolksdorf_R-sectoriality}. 
	\end{proof}
	
	   As usual, the $\tau$-dependent theory with $\tau = 0$ yields an analogous conclusion for the divergence-form operator $\cL$ in dimension $n+1$.
	
	\begin{corollary} \label{M(L_t) = P_+(L_t): Cor: M(cL) = P_+(cL)}
		The Meyers exponent for $\cL$ is given by
        \begin{align*}
            \Meyers(\cL) = \sup \Big\{p \geq 2: \nabla_{t,x} \cL^{-1} \Div_{t,x} \text{ is $\L^p$-bounded} \Big\}.
        \end{align*}
	\end{corollary}

            The supremum on the right-hand side in Corollary~\ref{M(L_t) = P_+(L_t): Cor: M(cL) = P_+(cL)} was studied extensively in \cite{Auscher-Egert_book}. Multiple characterizations are known and the one that turns out particularly useful for our purpose is as follows. 

        \begin{corollary} \label{M(L_t) = P_+(L_t): Cor: M(cL) = Stuff from AE}
		The Meyers exponent for $\cL$ is also given by
        \begin{align*}
            \Meyers(\cL) 
        = \sup \Big\{p \geq 2: \text{there is $N \in \N$ such that } (s \nabla_{t,x} (1+s^2 \cL)^{-N})_{s>0} \text{ is $\L^p$-bounded} \Big\}.
        \end{align*}
	\end{corollary}

            This is a combination of two results: By \cite[Thm.~13.12]{Auscher-Egert_book} one can always take $N=1$ and by \cite[Lem.~6.5]{Auscher-Egert_book} the choice of $N$ does not matter.
        
	
	\section{From $1+n$ to $n$ dimensions} \label{Sec: From cL to L_tau}

        In this section, we carry out the easiest step in Figure~\ref{fig: Roadmap} by showing that a Meyers estimate for $\cL$ implies a Meyers estimate for $L_\tau$ in one dimension lower, uniformly in $\tau$.
    
	\begin{proposition} \label{M(cL) < M(L_t): Prop: M(cL) < M(L_t)}
		The inequality $\Meyers(\cL) \leq \Meyers((L_\tau)_{\tau \in \R})$ holds true.
	\end{proposition}
	
	\begin{proof}
		Let $2 < p < \Meyers(\cL)$. Given $\tau \in \R$, $Q = Q(x, r) \sub \R^n$ and any $L_\tau$-harmonic $u$ in $2Q$, we need to show 
            \begin{align} \label{M(cL) < M(L_t): eq1}
			\left( \fint_Q |S_\tau u|^p \, \d x \right)^{\frac{1}{p}} \leq C 	\left( \fint_{2 Q} |S_\tau u|^2 \, \d x \right)^{\frac{1}{2}}.
		\end{align} 
        Splitting the coefficients of $\cL$ as in \eqref{eq: Splitting of A}, we find
        \begin{align*}
            \cL U = - \Div_x (d \nabla_x U) - \Div_x(c \partial_t U) - \partial_t (b \nabla_x U) - \partial_t (a \partial_t U).
        \end{align*}
        Comparing with \eqref{eq: Formula L-tau}, we see that $U(t, x) \coloneqq \e^{\i t \tau} u(x)$ defines an $\cL$-harmonic function in $(0, 4 r) \times 2 Q \subseteq \R^{1+n}$. The assumption yields some $C > 0$ not depending on $\tau$ such that 
		\begin{equation*}
			\left( \fint_r^{3r} \fint_Q |\nabla_{t,x} U|^p \, \d x \, \d t \right)^{\frac{1}{p}} \leq C \left( \fint_0^{4r} \fint_{2 Q} |\nabla_{t,x} U|^2 \, \d x \, \d t \right)^{\frac{1}{2}},
		\end{equation*} 
        but as $|\nabla_{t,x} U| = |S_\tau u|$ is independent of $t$, this estimate collapses to \eqref{M(cL) < M(L_t): eq1}.
	\end{proof}


	\section{From $n$ to $1+n$ dimensions} \label{Sec: From L_tau to cL}
	
	In this section, we prove the following reverse inequality to Proposition~\ref{M(cL) < M(L_t): Prop: M(cL) < M(L_t)} in dimension $n \geq 2$.
    
    \begin{proposition} \label{M(L_t) < M(cL): Prop: Goal}
        In dimension $n\geq 2$, the inequality $\Meyers(\cL) \geq \Meyers((L_\tau)_{\tau \in \R^*})$ holds true.
    \end{proposition}

    \begin{remark} \label{M(L_t) < M(cL): Rem: Goal}
        Since we have $\Meyers((L_\tau)_{\tau \in \R^*}) \geq \Meyers((L_\tau)_{\tau \in \R})$ by definition, combining Propositions~\ref{M(L_t) < M(cL): Prop: Goal} and \ref{M(cL) < M(L_t): Prop: M(cL) < M(L_t)} also reveals that $\Meyers((L_\tau)_{\tau \in \R^*}) = \Meyers((L_\tau)_{\tau \in \R})$.
    \end{remark}

    The proof of this estimate is the centerpiece of our paper and spreads over five subsections, following the strategy outlined in Figure~\ref{fig: Roadmap}.

    \subsection{\boldmath Reduction to a resolvent estimate in $1+n$ dimensions}

        For the rest of the section, we assume $n\geq2$ and use arbitrary exponents $r,q$ such that
    \begin{align}\label{Meyers meets DB: Assumptions p,q}
    \begin{split}
         2_* < r<2<q<\infty \quad \text{and} \quad \text{$(S_\tau L_\tau^{-1} S_\tau^*)_{\tau \in \R^*}$ is $\L^r$- and $\L^q$-bounded}.
    \end{split}
    \end{align}
    Such exponents exist thanks to Proposition~\ref{DB: Prop: I(DB) = Hodge((L_tau)_tau)} and the assumption is open-ended with respect to $r$ and $q$. (By this we mean that the same assumption holds for a smaller $r$ and a larger $q$.) The role of $r$ will become clear only at the very end of the proof, but it cannot be omitted. With these exponents at hand, we formulate a technical result from which Proposition~\ref{M(L_t) < M(cL): Prop: Goal} will follow, using the results from the previous sections.

    \begin{proposition} \label{M(L_t) < M(cL): Prop: 1st reduction}
    Let $q$ be as in \eqref{Meyers meets DB: Assumptions p,q}. For any $p \in (2,q)$ there exists $N \in \N$ such that $(s \nabla_{t,x} (1+s^2 \cL)^{-N})_{s>0}$ is $\L^p$-bounded.
    \end{proposition}

    \begin{proof}[\emph{Proof of Proposition~\ref{M(L_t) < M(cL): Prop: Goal}, admitting Proposition~\ref{M(L_t) < M(cL): Prop: 1st reduction}}]
    Let $p < \Meyers((L_\tau)_{\tau \in \R^*})$. By the characterization of $\Meyers((L_\tau)_{\tau \in \R^*})$ in Proposition~\ref{M(L_t) = P_+(L_t): Prop: M(L_t) = P_+(L_t)}, we have \eqref{Meyers meets DB: Assumptions p,q} with an exponent $q>p$ at our disposal. Proposition~\ref{M(L_t) < M(cL): Prop: 1st reduction} and the characterization of $\Meyers(\cL)$ in Corollary~\ref{M(L_t) = P_+(L_t): Cor: M(cL) = Stuff from AE} yield $p \leq \Meyers(\cL)$. Since $p < \Meyers((L_\tau)_{\tau \in \R^*})$ was chosen arbitrarily, the claim follows.
    \end{proof}
    
    \subsection{Reduction to an operator-valued multiplier estimate}

    In the following, we identify $\L^2(\R^{1+n}) \cong \L^2(\R; \L^2(\R^n)) \eqqcolon \L^2(\R; \L^2)$ via Fubini's theorem. Since the coefficients of $\cL$ are $t$-independent,  we  obtain the following correspondence to operator-valued Fourier multipliers.
	
	\begin{lemma} \label{M(L_t) < M(cL): Lem: symbol of (1 + s^2 cL)^-1}
		Let $s > 0$. Then $(1 + s^2 \cL)^{-1}$ is an $\cL(\L^2)$-valued Fourier multiplier with symbol 
		\begin{equation*}
			\tau \mapsto (1 + s^2 L_{\tau})^{-1}, 
		\end{equation*} 
		that is 
		\begin{equation*}
			(1 + s^2 \cL)^{-1} = \cF_t^{-1} (1 + s^2 L_{\tau})^{-1} \cF_t. 
		\end{equation*}
	\end{lemma}
	
	\begin{proof}
		We let $f \in \L^2(\R^{1 +n})$ and set $g \coloneqq (1 + s^2 \cL)^{-1} f \in \dom(\cL) \in \W^{1,2}(\R^{1+n})$. Then
		\begin{align*}
			f &=  g + s^2 \cL g
			\\&= g + s^2 (- \Div_x (d \nabla_x g) - \Div_x (c \partial_t g) - \partial_t (b \nabla_x g) - a \partial_t^2 g)
		\end{align*}
		in the weak sense. In particular, for all $h \in \cS(\R^{1+n})$, we have
		\begin{align*}
			\langle f, h \rangle &= \langle g, h \rangle + s^2 \Big( \langle d \nabla_x g, \nabla_x h \rangle + \langle c \partial_t g,  \nabla_x h \rangle + \langle b \nabla_x g, \partial_t h \rangle + \langle a \partial_t g, \partial_t h \rangle \Big).
		\end{align*}
		By Plancherel's theorem in the $t$-variable and since $a,b,c,d$ are $t$-independent, we get
		\begin{align*}
			\langle \cF_t f, \cF_t h \rangle &= \langle \cF_t g, \cF_t h \rangle 
			\\&\qquad + s^2 \Big( \langle d \nabla_x \cF_t g, \nabla_x \cF_t h \rangle + \langle \i \tau c \cF_t g, \nabla_x \cF_t h \rangle 
			\\&\qquad \qquad \quad  + \langle b \nabla_x \cF_t g, \i \tau cF_t h \rangle + \langle a \i \tau  \cF_t g, \i \tau \cF_t h \rangle \Big)
			\\&= \langle (1 + s^2 (- \Div_x d \nabla_x - \i \tau \Div_x c - \i \tau b \nabla_x + a \tau^2)) \cF_t g, \cF_t h \rangle.
		\end{align*}
		This calculation implies that $(\cF_t g)(\tau, \cdot) \in \dom(L_\tau)$ for a.e.\ $\tau \in \R$ with 
		\begin{equation*}
			(\cF_t f)(\tau, \cdot) = (1 + s^2 L_{\tau}) (\cF_t g)(\tau, \cdot).
		\end{equation*}
		Hence, 
		\begin{equation*}
			(1 + s^2 \cL)^{-1} f = g = \cF_t^{-1} (1 + s^2 L_{\tau})^{-1} \cF_t f. \qedhere 
		\end{equation*}
	\end{proof}

         In the following, we split vectors $f \in \C^{1+n}$ as $f = [f_\perp, f_\parallel]^\top$ in accordance with our writing for the coefficents $A$ in \eqref{eq: Splitting of A}.

    	\begin{corollary} \label{M(L_t) < M(cL): Cor: symbol of s grad (1 + s^2 cL)^-N}
		Let $s > 0$ and $N \in \N$. Then $s \nabla_{t, x} (1 + s^2 \cL)^{-N}$ is an $\cL(\L^2; (\L^2)^{1+n})$-valued Fourier multiplier with symbol
		\begin{equation} \label{eq: M(L) = M(L_tau): m_s(tau)}
			m_{s,N}(\tau) \coloneqq \begin{bmatrix}
				m_{s,N}(\tau)_\perp \\
				m_{s,N}(\tau)_\parallel
			\end{bmatrix} 
			\coloneqq \begin{bmatrix}
				\i s \tau (1 + s^2 L_{\tau})^{-N} \\ 
				s \nabla_x (1 + s^2 L_{\tau})^{-N}
			\end{bmatrix}.
		\end{equation}
	\end{corollary}

	   Proposition~\ref{M(L_t) < M(cL): Prop: 1st reduction} now asks for boundedness of a Fourier multiplier in $\L^p(\R; (\L^p)^{1+n})$. A sufficient condition is furnished by the celebrated theorem of Weis \cite{Weis_Lp-multipliers}. We phrase his result in terms of square function estimates rather than $\cR$-boundedness, which are equivalent concepts in the Banach space $\L^p$, see~\cite[Thm.~8.1.3~(3)]{Analysis_in_BS_II}.
	
	\begin{definition}
		Let $p \in (1, \infty)$ and $I \subseteq \R^*$. We say that $(T(t))_{t \in I} \subseteq \cL(\L^2)$ satisfies \emph{square function estimates} on $\L^p$, if there is $C \geq 0$ such that 
		\begin{equation*}
			\bigg\| \Big( \sum_{j=1}^k |T(t_j) f_j|^2 \Big)^{\frac{1}{2}} \bigg\|_p \leq C \bigg\| \Big( \sum_{j=1}^k |f_j|^2 \Big)^{\frac{1}{2}} \bigg\|_p
		\end{equation*}
		for all choices $k\in \N$, $t_1, \dots, t_k \in I$ and $f_1, \dots, f_k \in \L^p \cap \L^2$.
	\end{definition}

	\begin{theorem}[{Weis~\cite{Weis_Lp-multipliers}, \cite[Cor.~8.3.11]{Analysis_in_BS_II}}] \label{M(L_t) < M(cL): Thm: Weis_Lp-multiplier}
		Let $p \in (1, \infty)$ and $m \in \rC^1(\R^*; \cL(\L^p))$ be such that $(m(\tau))_{\tau \in \R^*}$ and $( \tau m'(\tau) )_{\tau \in \R^*}$ satisfy square function estimates on $\L^p$ with constant $C$. Then the operator given by
		\begin{equation*}
			T f = \cF_t^{-1} (m \cF_t f) \qquad (f \in \cS(\R^{1+n}))
		\end{equation*}
		has a unique bounded extension to $\L^p(\R; \L^p)$ with a bound depending on $p$ and $C$.
	\end{theorem}
    
        We can now further reduce Proposition~\ref{M(L_t) < M(cL): Prop: 1st reduction} to a square function estimate. Note carefully that the symbol in \eqref{eq: M(L) = M(L_tau): m_s(tau)} does not even map into $\L^p$ for general $p$ and already for this we will use properties of the $\tau$-dependent Hodge family. 

  	\begin{proposition} \label{M(L_t) < M(cL): Prop: SFE for m_s(tau)}
		Let $p \in (2,q)$. There exists $N \in \N$ such that for all $s > 0$ the symbol $m_{s,N}$ is of class $\rC^1(\R^*; \cL(\L^p))$, and $(m_{s,N}(\tau))_{\tau \in \R^*}$ and $(\tau m_{s,N}'(\tau))_{\tau \in \R^*}$ satisfy square function estimates on $\L^p$ with bounds independent of $s$. 
	\end{proposition}

        \begin{proof}[\emph{Proof of Propositions~\ref{M(L_t) < M(cL): Prop: Goal} \& \ref{M(L_t) < M(cL): Prop: 1st reduction}, admitting Proposition~\ref{M(L_t) < M(cL): Prop: SFE for m_s(tau)}}]
        Apply Theorem~\ref{M(L_t) < M(cL): Thm: Weis_Lp-multiplier} to the operators $s \nabla_{t, x} (1 + s^2 \cL)^{-N}$.
        \end{proof}

	\subsection{\boldmath $\L^2$-results for auxiliary operators}
	
	We introduce four auxiliary families of operators that will play a key role in the following. Our notation is easy to remember: the superscript on $R$ refers to the power of $s\tau$ and the subscript indicates the amount of $x$-derivatives left and right of the resolvents. 
	
	\begin{definition}  \label{M(L_t) < M(cL): Def: R families }
		For $s > 0$ and $\tau \in \R^*$, we set
		\begin{align*}
			\leftindex^\alpha{R}_{(0,0)}(s, \tau) &\coloneqq (s \tau)^\alpha (1 + s^2 L_{\tau})^{-1} \qquad ( \alpha \in \{ 0, 1, 2 \}), \\
			\leftindex^\alpha R_{(1, 0)}(s, \tau) &\coloneqq (s \tau)^\alpha s \nabla_x (1 + s^2 L_{\tau})^{-1} \qquad (\alpha \in \{ 0, 1\}), \\
			\leftindex^\alpha R_{(0,1)}(s, \tau) &\coloneqq (s \tau)^\alpha (1 + s^2 L_{\tau})^{-1} s \Div_x \qquad (\alpha \in \{ 0, 1 \}), \\
			\leftindex^\alpha R_{(1, 1)}(s, \tau) &\coloneqq (s \tau)^\alpha s \nabla_x (1 + s^2 L_{\tau})^{-1} s \Div_x \qquad (\alpha = 0). 
		\end{align*}	 
	\end{definition}

        We need $\L^2$ off-diagonal estimates for these families with respect to the parameter $s$ and implicit constants that do not depend on $\tau$. To this end, we rely on scaling. The rescaled coefficients $A(\tau^{-1} \cdot)$ have the same $\L^\infty$-norm and satisfy our ellipticity assumption from Section~\ref{Sec: L2 theory} with the same constant $\lambda$. From \eqref{eq: Formula L-tau} we obtain that the corresponding inhomogeneous divergence form operator is given by 
	\begin{align}\label{eq: Formula L-tau rescaled}
        \begin{split}
		L^{A(\tau^{-1} \cdot)}u 
        &\coloneqq -\Div_x (d(\tau^{-1} \cdot) \nabla_x u) - \i \Div_x(c(\tau^{-1} \cdot)u) - \i b(\tau^{-1} \cdot) \nabla_x u + a(\tau^{-1} \cdot) u
        \\&=\tau^{-2} \delta_{\tau^{-1}} L_{\tau} \delta_{\tau}. 
        \end{split}
	\end{align} 
        The following $\L^2$ off-diagonal estimates for semigroups generated by inhomogeneous operators in divergence form are standard nowadays. For an explicit statement of (i), see \cite[Prop.~3.2]{Bechtel_Lp}. Statement (ii) follows by general duality and composition principles, see \cite[Sec.~4.2]{Auscher-Egert_book}.
	
	\begin{lemma}  \label{M(L_t) < M(cL): Lem: L2 ODE for sg operators} 
		There are $C, c > 0$ depending only on $\lambda$ and $\| A \|_\infty$ such that for all $s > 0$, $\tau \in \R^*$, measurable sets $E, F \sub \R^n$ and $f \in \L^2$ with support in $E$, we have the following off-diagonal bounds:
        \begin{enumerate}
            \item $\displaystyle \| \e^{- s^2 L^{A(\tau \cdot)}} f \|_{\L^2(F)} + \| s \nabla_x \e^{- s^2 L^{A(\tau \cdot)}} f \|_{\L^2(F)} \leq C \e^{- c^2 \frac{\d(E, F)^2}{s^2} - c s^2} \| f \|_{\L^2(E)}$,
            \item $\displaystyle \| \e^{- s^2 L^{A(\tau \cdot)}} s \Div_x f \|_{\L^2(F)} + \| s \nabla_x \e^{- s^2 L^{A(\tau \cdot)}} s \Div_x f \|_{\L^2(F)} \\  
             {}\hfill \leq C \e^{- c^2 \frac{\d(E, F)^2}{s^2} - c s^2} \| f \|_{\L^2(E)}$.
        \end{enumerate}
	\end{lemma}

        In a next step, we take the Laplace transform to derive estimates for resolvents. For technical reasons, we need an explicit $\L^2$-bound for $\leftindex^{0} R_{(1,1)}(s, \tau)$ beforehand.

	\begin{lemma} \label{M(L_t) < M(cL): Lem: L2 bound for R_(1,1)}
		For all $s > 0$, $\tau \in \R^*$ and $f \in \L^2$, we have
		\begin{equation*}
			\|\leftindex^0 R_{(1,1)} (s, \tau) f \|_2 \leq \lambda^{-1} \| f \|_2.
		\end{equation*}
	\end{lemma}
	
	\begin{proof}
		We set $u \coloneqq  (1 + s^2 L_\tau)^{-1} \Div_x f$. By ellipticity,
		\begin{equation*}
			\lambda s^2 \| u \|_{\W^{1,2}_\tau}^2 \leq \re \langle (1 + s^2 L_\tau) u, u \rangle \leq \| \Div_x f \|_{\W^{-1,2}_\tau} \| u \|_{\W^{1,2}_\tau},
		\end{equation*}
		and therefore
		\begin{align*}
			\|\leftindex^{0} R_{(1,1)} (s,\tau)f\|_2 = s^2 \|\nabla_x u\|_2 &\leq s^2 \| u \|_{\W^{1,2}_\tau} 
			\leq \lambda^{-1} \| \Div_x f \|_{\W^{-1,2}_\tau} 
			\leq \lambda^{-1} \| f \|_2.  \qedhere
		\end{align*}
	\end{proof}

	\begin{lemma} \label{M(L_t) < M(cL): Lem: L2 ODE for R families}
		Let $\beta \in \{ 0, 1 \}^2$, $0\leq \alpha \leq  2- |\beta|$ and $\tau \in \R^*$. Then $(\leftindex^\alpha R_\beta(s, \tau))_{s>0}$ satisfies $\L^2$ off-diagonal estimates with implicit constants depending only on $\lambda$ and $\| A \|_\infty$.
	\end{lemma}
	
	\begin{proof}
		We take $s > 0$, $\tau \in \R^*$, measurable sets $E, F \subseteq \R^n$, and $f \in \L^2$ supported in $E$ and normalized to $\| f \|_{\L^2(E)} = 1$. In view of Lemma~\ref{M(L_t) < M(cL): Lem: L2 bound for R_(1,1)}, we can additionally assume that $\nicefrac{\d(E, F)}{s} \geq 1$ if $\beta = (1,1)$. Let us write
            \begin{equation*}
                \leftindex^\alpha{R}_\beta^{A(\tau^{-1} \cdot)}(s, 1)
		\end{equation*} 
		for the operator $\leftindex^\alpha R_\beta(s,1)$ with coefficients $A(\tau^{-1} \cdot)$ in place of $A$. By the Laplace transform formula 
        \begin{align*}
            (1+s^2L^{A(\tau^{-1} \cdot)})^{-1} f = \int_0^\infty \e^{-t} \e^{- t s^2 L^{A(\tau^{-1} \cdot)}} f \, \d t
        \end{align*}
        and Lemma~\ref{M(L_t) < M(cL): Lem: L2 ODE for sg operators}, we get 
		\begin{align*}
			\| &\leftindex^\alpha{R}_\beta^{A(\tau^{-1} \cdot)}(s, 1)f \|_{\L^2(F)} 
                \\&\leq C s^{\alpha} \int_0^{\infty} \e^{-t - c t s^2 - c^2 \frac{\d(E, F)^2}{t s^2}} \, \frac{\d t}{t^{\nicefrac{|\beta|}{2}}} 
			\\&\leq C s^\alpha \int_0^{\infty} \e^{- \frac{t}{2} - c t s^2 - \left(\frac{t}{2} + c^2 \frac{\d(E, F)^2}{2 t s^2} \right) - c^2 \frac{\d(E, F)^2}{2 t s^2}}  \, \frac{\d t}{t^{\nicefrac{|\beta|}{2}}}
                \\&\leq C \e^{- c \frac{\d(E, F)}{s}} s^\alpha \int_0^{\infty} \e^{- \frac{1}{2}(1 + 2 c s^2) t} \e^{- c^2 \frac{\d(E, F)^2}{t s^2}} t^{- \frac{|\beta|}{2}} \, \d t,
			\intertext{where we used the inequality $2 |x y| \leq x^2 + y^2$ in the final step. Now, we invoke the assumption $\nicefrac{\d(E, F)}{s} \geq 1$ when $\beta = (1, 1)$ and substitute $u = (1+ 2 cs^2)t$ in order to arrive at}
			&\leq C \e^{- c \frac{\d(E, F)}{s}} s^\alpha \int_0^{\infty} \e^{- \frac{1}{2}(1 + 2c s^2)t} \Big(\1_{[|\beta| < 2]} t^{- \frac{|\beta|}{2}} + \1_{[|\beta| =2]}\Big)  \, \d t
			\\&= C \e^{- c \frac{\d(E, F)}{s}} \int_0^{\infty} \e^{- \frac{1}{2}u} \Big(\1_{[|\beta| < 2]} s^\alpha (1+2cs^2)^{\frac{|\beta|}{2}-1} u^{-\frac{|\beta|}{2}} + \1_{[|\beta| =2]}(1+2cs^2)^{-1}\Big)  \, \d u.
		\end{align*}
            The terms in $s$ are uniformly bounded due to the restriction $0 \leq \alpha \leq 2 - |\beta|$ and the remaining integral in $u$ is finite, leading to an overall estimate
            \begin{align}\label{M(L_t) < M(cL): eq1: L2 ODE for R families}
                \|\leftindex^\alpha{R}_\beta^{A(\tau^{-1} \cdot)}(s, 1)f \|_{\L^2(F)} \leq C \e^{- c \frac{\d(E, F)}{s}}.
            \end{align}
            
		We finish the proof by scaling. By similarity as in \eqref{eq: Formula L-tau rescaled}, we have 
		\begin{equation*}
  \sgn(\tau)^{\alpha + |\beta|} \leftindex^\alpha{R}_\beta^{A(\tau^{-1} \cdot)}(s |\tau|, 1) \delta_{\tau^{-1}} = \delta_{\tau^{-1}} \leftindex^\alpha R_\beta (s, \tau). 
		\end{equation*}
		Hence, \eqref{M(L_t) < M(cL): eq1: L2 ODE for R families}, when read with $(\delta_{\tau^{-1}}f, \tau E, \tau F)$ in place of $(f, E, F)$, becomes
		\begin{equation*}
			\| \delta_{\tau^{-1}} \leftindex^\alpha R_\beta (s, \tau) f \|_{\L^2(\tau F)} \leq C \e^{- c \frac{\d(\tau E, \tau F)}{s |\tau|}} \| \delta_{\tau^{-1}} f \|_{\L^2(\tau E)}
		\end{equation*}
		and by the transformation rule we conclude the required off-diagonal bound
        \begin{equation*}
			\| \leftindex^\alpha R_\beta (s, \tau) f \|_{\L^2(F)} \leq C \e^{- c \frac{\d( E, F)}{s}} \| f \|_{\L^2(E)}. \qedhere
		\end{equation*}
	\end{proof}
	
	\subsection{\boldmath $\L^p$-theory for the auxiliary operators}
	
	We start with the following Sobolev embeddings related to the resolvent. The qualitative result is not new but once again $\tau$-independence of constants matters.
	
	\begin{lemma}  \label{M(L_t) < M(cL): Lem: Lp bound for R_(0,0)}
		If $2_* < p < q^*$, then $(\leftindex^{0}R_{(0,0)}(s, \tau))_{s>0}$ is $\L^p$-bounded with implicit constant independent of $\tau \in \R^*$.
	\end{lemma}
	
	\begin{proof}
            We use a bootstrapping argument to increase the exponent $p$ step-by-step.
		
		\emph{Base case $\boldsymbol{2_* < p < 2^*}$.} We first consider exponents $p \geq 2$. By the Gagliardo--Nirenberg inequality and the $\L^2$-estimates from Lemma~\ref{M(L_t) < M(cL): Lem: L2 ODE for R families}, we find for all $s>0$ and $f \in \L^2$ that
        \begin{align*}
            \| (1 + s^2 L_\tau)^{-1} f \|_p &\les s^{\frac{n}{p} - \frac{n}{2}} \| s \nabla_x (1 + s^2 L_\tau)^{-1} f \|_2^{\frac{n}{2} - \frac{n}{p}} \| (1 + s^2 L_\tau)^{-1} f \|_2^{1 - \frac{n}{2} + \frac{n}{p}}
            \\&\les s^{\frac{n}{p} - \frac{n}{2}} \| f \|_2.
        \end{align*}
        This means that $(\leftindex^{0}R_{(0,0)}(s, \tau))_{s>0}$ is $\L^2 - \L^p$-bounded. Since the restriction on $p$ is open-ended, $\L^p$-boundedness follows from Lemma~\ref{Lp-Lq estimates: Lem: Lp -> Lq + L2 ODE implies Lr -> Ls ODE}. 
        
        So far, we have only used $\L^2$-theory for $L_\tau$. Hence, the same conclusion is valid for the adjoint $L_\tau^*$, which is an operator in the same class as $L_\tau$ and by duality $(\leftindex^{0}R_{(0,0)}(s, \tau))_{s>0}$ is $\L^{p'} - \L^2$-bounded and $\L^{p'}$-bounded. This covers all exponents in $(2_*,2)$.
		
		\emph{Inductive case.} The next case to consider is $2^* < p < ((2^{*})^* \wedge q^*)$. (We do not have to consider $p=2^*$ explicitly since all assumptions on $p$ are open-ended). This scenario can only appear in dimension $n \geq 3$ since $2^* = \infty$ when $n=2$. 
        
        Since $(p_*)_* > 2_*$, we know from the base case that $(\leftindex^{0}R_{(0,0)}(s, \tau))_{s>0}$ is $\L^{(p_*)_*}$-bounded. Given $f \in \L^{(p_*)_*} \cap \L^2$, Sobolev embeddings yields
		\begin{align*}
			\| (1 + s^2 L_{\tau})^{-1} f \|_p &\les s^{-2} \| \nabla_x L_{\tau}^{-1} (1 - (1 + s^2 L_{\tau})^{-1}) f \|_{p_*}
			\\&\leq s^{-2} \| L_{\tau}^{-1} (1 - (1 + s^2 L_{\tau})^{-1}) f \|_{\W^{1, p_*}_{\tau}}.
			\intertext{Since $2 < p_* < q$, the $\L^{p_*}$-boundedness of the Hodge projectors in its equivalent form in Proposition~\ref{Inhom Hodge theory: Prop: Hodge-range via Hodge} and another Sobolev embedding lead to}
			&\les  s^{-2} \| (1 - (1 + s^2 L_{\tau})^{-1}) f \|_{\W^{-1, p_*}_{\tau}}
			\\&\les s^{-2} \| (1 - (1 + s^2 L_{\tau})^{-1}) f \|_{(p_*)_*}.
		\end{align*} 
		This means that $(\leftindex^{0}R_{(0,0)}(s, \tau))_{s>0}$ is $\L^{(p_*)_*}-\L^p$-bounded and Lemma~\ref{Lp-Lq estimates: Lem: Lp -> Lq + L2 ODE implies Lr -> Ls ODE} yields the desired $\L^p$-boundedness as before. All implicit constants in this argument are independent of $\tau$.

        Iterating the procedure covers the full range $2^* < p < q^*$ in a finite number of steps.
	\end{proof}

        Our proof in the base case revealed an additional result that we record for later.

        \begin{corollary}\label{M(L_t) < M(cL): Cor: Lp bound for R_(0,0)}
        If $2_* < p < 2$, then $(\leftindex^{0}R_{(0,0)}(s, \tau))_{s>0}$ is $\L^p - \L^2$-bounded with implicit constant independent of $\tau \in \R^*$.
        \end{corollary}

        We turn to the other operator families.
	
	\begin{lemma} \label{M(L_t) < M(cL): Lem: Lp bound for rest}
		Let $|\beta| \leq 1$ and $\alpha = 1$ when $\beta = (0, 1)$. Then $(\leftindex^\alpha R_\beta (s, \tau))_{s>0}$ is $\L^r$- and $\L^q$-bounded with implicit constant independent of $\tau \in \R^*$. 
	\end{lemma}
	
	\begin{proof}
		Since the assumption on $r$ and $q$ in~\eqref{Meyers meets DB: Assumptions p,q} is open-ended, it will be enough to prove a respective $\L^p$-bound for $p \in (r,q)$.  For $\leftindex^{0}R_{(0,0)}$ this is Lemma~\ref{M(L_t) < M(cL): Lem: Lp bound for R_(0,0)}. The other families can be classified into three groups.
		
		\emph{The family $\boldmath{\leftindex^\alpha R_{(0,0)}}$  with $\boldsymbol{\alpha \in \{ 1, 2\}}$.} Since $|s\tau| \leq 1 +|s \tau|^2$, it suffices to treat $\leftindex^2 R_{(0,0)}$. However,
		\begin{equation*}
			(s \tau)^2 (1 + s^2 L_{\tau})^{-1} = \tau^2 L_{\tau}^{-1} (1 - (1 + s^2 L_{\tau})^{-1}),
		\end{equation*}
		and the claim follows by combining Proposition~\ref{Inhom Hodge theory: Prop: Hodge-range via Hodge} and Lemma~\ref{M(L_t) < M(cL): Lem: Lp bound for R_(0,0)}. 

        	\emph{The families $\boldmath{ \leftindex^1 R_{(1,0)}}$ and $\boldmath{\leftindex^1 R_{(0,1)}}$.} We write 
		\begin{align*}
			(s \tau) s \nabla_x (1 + s^2 L_{\tau})^{-1} 
            &= \tau \nabla_x L_{\tau}^{-1} (1 - (1 + s^2 L_{\tau})^{-1}), 
            \\
                (s \tau) (1 + s^2 L_{\tau})^{-1} s \Div_x 
            &= (1 - (1 + s^2 L_{\tau})^{-1}) \tau L_{\tau}^{-1} \Div_x
		\end{align*}
		and conclude once again by Proposition~\ref{Inhom Hodge theory: Prop: Hodge-range via Hodge} and Lemma~\ref{M(L_t) < M(cL): Lem: Lp bound for R_(0,0)}.
		
		\emph{The family $\boldmath{\leftindex^{0} R_{(1,0)}}$.} Since we have $q>2$ and $r>2_*$, it is suffices to show that for $p \in (2,q)$ both families are $\L^p$- and $\L^{p_*}$-bounded with $\tau$-independent bound. Note that in this case $p_* > 2_* \geq 1$. As before, we write
		\begin{equation*}
			s \nabla_x (1 + s^2 L_{\tau})^{-1} = s^{-1} \nabla_x L_{\tau}^{-1} (1 - (1 + s^2 L_{\tau})^{-1}).
		\end{equation*}
            Proposition~\ref{Inhom Hodge theory: Prop: Hodge-range via Hodge}, the Sobolev embedding $\L^{p_*} \sub \W^{-1, p}_{\tau}$ and Lemma~\ref{M(L_t) < M(cL): Lem: Lp bound for R_(0,0)} yield for all $f \in 
        \L^{p_*} \cap \L^2$ the bound
		\begin{align*}
			\| s \nabla_x (1 + s^2 L_{\tau})^{-1} f \|_p &\leq s^{-1} \| L_{\tau}^{-1} (1 - (1 + s^2 L_{\tau})^{-1}) f \|_{\W^{1, p}_{\tau}}
			\\&\les s^{-1} \| (1 - (1 + s^2 L_{\tau})^{-1}) f \|_{\W^{-1, p}_{\tau}}
			\\&\les s^{-1} \| (1 - (1 + s^2 L_{\tau})^{-1}) f \|_{p_*}
			\\&\les s^{-1} \| f \|_{p_*}. 
		\end{align*}
		This means that $(\leftindex^{0}R_{(1,0)}(s, \tau))_{s>0}$ is $\L^{p_*} - \L^p$-bounded with $\tau$-independent bound. Open-endedness in $p$ and Lemma~\ref{Lp-Lq estimates: Lem: Lp -> Lq + L2 ODE implies Lr -> Ls ODE} yield the claim.
	\end{proof}

        \subsection{\boldmath $\L^p$-theory for the symbol}
        
	In this section, we assemble estimates for the auxiliary functions in order to derive smoothness and $\L^p - \L^q$-type bounds for the symbol $m_{s,N}(\tau)$ from \eqref{eq: M(L) = M(L_tau): m_s(tau)}. 

	\begin{lemma} \label{M(L_t) < M(cL): Lem: derivative of resolvents}
		Let $s > 0$ and $N \in \N$. If $r<p<q$, then $\tau \mapsto (1 + s^2 L_{\tau})^{-N}$ is of class $\rC^1(\R^*; \cL(\L^p, \W^{1,p}_1))$ with derivative 
		\begin{equation*}
			\frac{\d}{\d \tau} (1 + s^2 L_{\tau})^{-N} = s^2 \sum_{k = 1}^N (1 + s^2 L_{\tau})^{- k} ( \i \Div_x c + \i b \nabla_x - 2 \tau a ) (1 + s^2 L_{\tau})^{- (N + 1 - k)}.
		\end{equation*}
            In particular, we have $m_{s,N} \in \rC^1(\R^*; \cL(\L^p))$ for every $s > 0$.
	\end{lemma}
	
	\begin{proof}
 	We begin with the assertion about the resolvent. By the product rule and induction, it suffices to do the case $N=1$.
    
        For $\sigma \neq \tau$, we have 
    	\begin{align}
            \begin{split} \label{eq: M(L_t) < M(cL): resolvent diff}
    		&(1 + s^2 L_\tau)^{-1} - (1 + s^2 L_\sigma)^{-1} 
    		\\&= \, s^2 (1 + s^2 L_\tau)^{-1} (L_\sigma - L_\tau)(1 + s^2 L_\sigma)^{-1} 
    		\\&= \, s^2 (1 + s^2 L_\tau)^{-1} (\i (\tau - \sigma) \Div_x c + \i (\tau - \sigma) b \nabla_x - (\tau^2 - \sigma^2) a)(1 + s^2 L_\sigma)^{-1} 
    		\\&= \, (\tau - \sigma) s^2 (1 + s^2 L_\tau)^{-1} (\i \Div_x c + \i b \nabla_x - (\tau + \sigma) a)(1 + s^2 L_\sigma)^{-1} 
                \\&\eqqcolon (\tau - \sigma) r(\sigma, \tau).
            \end{split}
    	\end{align}
        Thanks to Lemma~\ref{M(L_t) < M(cL): Lem: Lp bound for rest}, the remainder $r(\sigma, \tau)$ is bounded in $\cL(\L^p)$-norm, uniformly in $\sigma$ and $\tau$ in compact subsets of $\R^*$. Hence, $\tau \mapsto (1 + s^2 L_\tau)^{-1}$ is continuous with values in $\cL(\L^p)$. 
    
        The same type of argument can be used to prove continuity with values in $\cL(\L^p)$ for $\tau \mapsto \nabla_x (1 + s^2 L_\tau)^{-1}$ and $\tau \mapsto (1 + s^2 L_\tau)^{-1} \Div_x$. In this calculation, the operator $s^2 \nabla_x (1 + s^2 L_\tau)^{-1} \Div_x$ appears, which cannot be handled via  Lemma~\ref{M(L_t) < M(cL): Lem: Lp bound for rest}. However, we can write
    	\begin{equation*}
    		s^2 \nabla_x (1 + s^2 L_\tau)^{-1} \Div_x = \nabla_x L_\tau^{-1} \Div_x - \nabla_x L_\tau^{-1} (1 + s^2 L_\tau)^{-1} \Div_x, 
    	\end{equation*}
        as a composition of operators that either fall in the scope of Lemma~\ref{M(L_t) < M(cL): Lem: Lp bound for rest} or are controlled through the Hodge projector. 
        
        Altogether, the remainder function $\tau \mapsto r(\sigma, \tau)$ in \eqref{eq: M(L_t) < M(cL): resolvent diff} is continuous with values in $\cL(\L^p, \W^{1,p}_1)$. Thus, $\sigma \mapsto (1+s^2 L_\sigma)^{-1}$ is of class $\rC^1(\R^*; \cL(\L^p, \W^{1,p}_1))$ with derivative 
        \begin{align*}
            \frac{\d}{\d \sigma} (1 + s^2 L_\sigma)^{-1} = r(\sigma, \sigma)
        \end{align*}
        as claimed.
    	
    	Continuous differentiability of $m_{s,N}$ follows immediately by the product rule since we have $m_{s,N}(\tau)_\perp = \i s \tau (1 + s^2 L_\tau)^{-N}$ and $m_{s,N}(\tau)_\parallel = s \nabla_x (1 + s^2 L_\tau)^{-N}$.
	\end{proof}

        We have reached the point in the argument, where we choose $N$ large.

        \begin{lemma}\label{M(L_t) < M(cL): Lem: Choice of N}
        There exists $N \in \N$, divisible by $4$, such that $((1 + s^2 L_{\tau})^{-\nicefrac{N}{4}})_{s > 0}$ is $\L^r - \L^q$-bounded with implicit constant independent of $\tau \in \R^*$.
        \end{lemma}

        \begin{proof}
            We consider the resolvent family $((1 + s^2 L_\tau)^{-1})_{s > 0}$ and the following boundedness properties with implicit constant independent of $\tau \in \R^*$. By Lemma~\ref{M(L_t) < M(cL): Lem: Lp bound for R_(0,0)}, we have $\L^\varrho$-bounds for all $\varrho \in (q, q^*)$ and by Corollary~\ref{M(L_t) < M(cL): Cor: Lp bound for R_(0,0)} we have $\L^\sigma - \L^2$-bounds for all $\sigma \in (2_*,r)$. By interpolation, we get an $\L^\sigma - \L^q$ bound for some $\sigma$. Again by Lemma~\ref{M(L_t) < M(cL): Lem: Lp bound for R_(0,0)} we also have the $\L^\sigma$- and the $\L^q$-bound. Now, Lemma~\ref{Lp-Lq estimates: Triangle Argument} yields the claim.
        \end{proof}

        The proof of the next result clarifies why this choice is appropriate for our purpose.
	
	\begin{lemma} \label{M(L_t) < M(cL): Lem: Lr-Lq bound for m_s(tau)}
		If $N$ is as in Lemma~\ref{M(L_t) < M(cL): Lem: Choice of N}, then  $(m_{s,N}(\tau))_{s > 0}$ and $(\tau m_{s,N}'(\tau))_{s > 0}$ are $\L^r - \L^q$-bounded with implicit constants independent of $\tau \in \R^*$. 
	\end{lemma}
	
	\begin{proof}

        The bound for $m_{s,N}(\tau)$ follows by composition from Lemma~\ref{M(L_t) < M(cL): Lem: Choice of N} and the $\L^q$-bounds for $\leftindex^1 R_{(0,0)}$ and $\leftindex^{0} R_{(1,0)}$ in Lemma~\ref{M(L_t) < M(cL): Lem: Lp bound for rest}. Let us come to the estimates for $\tau m_{s, N}'(\tau)$.
		
		\emph{The scalar component.} We compute 
		\begin{equation*}
			\tau m_{s,N}'(\tau)_\perp= m_{s,N}(\tau)_\perp + \i s \tau^2 \frac{\d}{\d \tau} (1 + s^2 L_{\tau})^{-N} 
		\end{equation*}
            and still need to handle the second term on the right. To this end, we write the formula in Lemma~\ref{M(L_t) < M(cL): Lem: derivative of resolvents} in the following form:
            \begin{align}\label{M(L_t) < M(cL): eq1: Lr-Lq bound for m_s(tau)}
                \i s \tau^2 \frac{\d}{\d \tau} (1 + s^2 L_{\tau})^{-N} 
                = \sum_{k = 1}^N (1 + s^2 L_{\tau})^{-(k-1)} T(s,\tau) (1 + s^2 L_{\tau})^{- (N - k)},
            \end{align}
            where 
		\begin{align}\label{M(L_t) < M(cL): T(s,tau)-family}
            \begin{split}
			T(s,\tau) \coloneqq&\, \i s^3 \tau^2 (1 + s^2 L_{\tau})^{-1} (\i \Div_x c + \i b \nabla_x - 2 \tau a) (1 + s^2 L_{\tau})^{-1}
			\\ =& -\Big((s \tau) (1 + s^2 L_{\tau})^{-1} s \Div_x \Big)\Big(c (s \tau) (1 + s^2 L_{\tau})^{-1}\Big)
			\\ & -\Big((s \tau)^2 (1 + s^2 L_{\tau})^{-1}\Big) \Big( b s \nabla_x (1 + s^2 L_{\tau})^{-1}\Big)
			\\ &- 2 \i \Big((s \tau) (1 + s^2 L_{\tau})^{-1}\Big) \Big(a (s \tau)^2 (1 + s^2 L_{\tau})^{-1}\Big)
            \end{split}
		\end{align}
		is a composition of operators  $\leftindex^\alpha R_\beta(s,\tau)$ with $|\beta| \leq        1$ and $\alpha=1$ for $\beta = (0,1)$. Thus, Lemma~\ref{M(L_t) < M(cL): Lem: Lp bound for rest} yields that $(T(s,\tau))_{s>0}$ is $\L^r$- and $\L^q$-bounded, independently of $\tau$. The upshot is that for each summand in \eqref{M(L_t) < M(cL): eq1: Lr-Lq bound for m_s(tau)} the exponents sum up to $(k-1) + (N-k) = N-1 \geq \nicefrac{N}{2}$ and hence one exponent is at least $\nicefrac{N}{4}$. By the choice of $N$, the entire expression, as a family indexed in $s>0$, is $\L^r-\L^q$-bounded. This concludes the treatment of the scalar component.
		
		\emph{The vectorial component.} Since $\nabla_x: \W^{1,p}_1 \to (\L^p)^n$ is bounded, we obtain from Lemma~\ref{M(L_t) < M(cL): Lem: derivative of resolvents} and with $T(s,\tau)$ as in \eqref{M(L_t) < M(cL): T(s,tau)-family} that
		\begin{align*}
			\tau m_{s,N}'(\tau)_\parallel 
            &= s \tau \nabla_x \frac{\d}{\d \tau} (1 + s^2 L_{\tau})^{-N}
                        \\&= \sum_{k = 1}^N s \nabla_x (1 + s^2 L_{\tau})^{-(k-1)} (\i s \tau)^{-1} T(s,\tau) (1 + s^2 L_{\tau})^{- (N - k)}
            \\&= \sum_{k = 2}^N \Big(s \nabla_x (1 + s^2 L_{\tau})^{-1}\Big) (1 + s^2 L_{\tau})^{-(k-2)} \Big((\i s \tau)^{-1} T(s,\tau)\Big)(1 + s^2 L_{\tau})^{- (N - k)}
            \\&\quad + \Big(s \nabla_x (\i s \tau)^{-1} T(s,\tau)\Big) (1 + s^2 L_{\tau})^{- (N - 1)}.
		\end{align*}
        Let us first handle the terms in $k = 2,\dots,N$. It follows from \eqref{M(L_t) < M(cL): T(s,tau)-family} that $((\i s \tau)^{-1} T(s,\tau))_{s>0}$ is a composition of the same type as $(T(s,\tau))_{s>0}$. Hence, this family is $\L^r$- and $\L^q$-bounded, independently of $\tau$. By Lemma~\ref{M(L_t) < M(cL): Lem: Lp bound for rest}, the same is true for $(s \nabla_x (1 + s^2 L_{\tau})^{-1})_{s>0}$ appearing on the left. The sum of the exponents still satisfies $(k-2) + (N-k) = N-2 \geq \nicefrac{N}{2}$, so we obtain the required $\L^r - \L^q$-bound as before.

        For the final term, we already know that $((1 + s^2 L_{\tau})^{- (N - 1)})_{s>0}$ is $\L^r-\L^q$-bounded with implicit constants independent of $\tau \in \R^*$. Thus, it suffices to prove $\L^q$-boundedness, uniformly in $\tau$, for
	\begin{align} \label{M(L_t) < M(cL): eq2: Lr-Lq bound for m_s(tau)}
        \begin{split}
			s \nabla_x (\i s \tau)^{-1}  T(s,\tau)
			=& \, \i \Big( s \nabla_x (1 + s^2 L_{\tau})^{-1} s \Div_x\Big) \Big(c (s \tau)(1 + s^2 L_{\tau})^{-1}\Big)
			\\ &+ \i \Big((s \tau) s \nabla_x (1 + s^2 L_{\tau})^{-1} \Big) \Big(b s \nabla_x (1 + s^2 L_{\tau})^{-1}\Big)
			\\ & - 2 \Big(s \nabla_x (1 + s^2 L_{\tau})^{-1} \Big) \Big(a (s \tau)^2 (1 + s^2 L_{\tau})^{-1}\Big).
        \end{split}
	\end{align}
		The second and third term on the right are $\L^q$-bounded by Lemma~\ref{M(L_t) < M(cL): Lem: Lp bound for rest} but this lemma does not cover the family $\leftindex^{0} R_{(1,1)}$ that appears in the first term. However, writing
		\begin{align*}
			s \nabla_x (1 + s^2 L_{\tau})^{-1} s \Div_x &= \nabla_x L_{\tau}^{-1} s^2 L_{\tau} (1 + s^2 L_{\tau})^{-1} \Div_x
			\\&= \nabla_x L_{\tau}^{-1} \Div_x - s \nabla_x (1 + s^2 L_{\tau})^{-1} \tau L_{\tau}^{-1} \Div_x (s\tau)^{-1},
		\end{align*}
		we see that the first of the three terms on the right in \eqref{M(L_t) < M(cL): eq2: Lr-Lq bound for m_s(tau)} can also be decomposed into
		\begin{align*}
		 \i \Big(\nabla_x L_{\tau}^{-1} \Div_x \Big) \Big(c (s \tau) (1 + s^2 L_{\tau})^{-1}\Big)  - \i \Big(s \nabla_x (1 + s^2 L_{\tau})^{-1} \Big)\Big(\tau L_{\tau}^{-1} \Div_x \Big)\Big(c (1 + s^2 L_{\tau})^{-1}\Big).			
		\end{align*}
		Now, the $\L^q$-boundedness follows from Proposition~\ref{Inhom Hodge theory: Prop: Hodge-range via Hodge} and Lemma~\ref{M(L_t) < M(cL): Lem: Lp bound for rest}.
	\end{proof}

        The uniform $\L^r-\L^q$-bounds can be upgraded to off-diagonal estimates upon possibly taking $N$ even larger.
	
	\begin{proposition} \label{M(L_t) < M(cL): Prop: Lr-Lq ODE for m_s(tau)}
		There exists a positive integer $N$ such that the families $(m_{s,N}(\tau))_{s > 0}$ and $(\tau m_{s,N}'(\tau))_{s > 0}$ satisfy $\L^r - \L^q$ off-diagonal estimates with implicit constants independent of $\tau \in \R^*$.
	\end{proposition}

        \begin{proof}
            Since the assumptions on $r$ and $q$ are open-ended, it suffices to prove $\L^\sigma-\L^p$ off-diagonal estimates whenever $r < \sigma < p < q$. 
            
            We already know that $(m_{s,N}(\tau))_{s > 0}$ and $(\tau m_{s,N}'(\tau))_{s > 0}$ are $\L^r-\L^q$-bounded. They also satisfy $\L^2$ off-diagonal estimates with implicit constants independent of $\tau$. Indeed, we have seen in the proof of Lemma~\ref{M(L_t) < M(cL): Lem: Lr-Lq bound for m_s(tau)} that $(m_{s,N}(\tau))_{s>0}$ and $(\tau m_{s,N}'(\tau))_{s>0}$ can be written as a sum and composition of the auxiliary families $(\leftindex^\alpha R_\beta(s,\tau))_{s>0}$, which satisfy $\L^2$ off-diagonal estimates by Lemma~\ref{M(L_t) < M(cL): Lem: L2 ODE for R families}. Thus, the claim is a consequence of Lemma~\ref{Lp-Lq estimates: Lem: Lp -> Lq + L2 ODE implies Lr -> Ls ODE}. 
        \end{proof}

    \subsection{Square function bounds for the symbol}
    
	Off-diagonal estimates as in Proposition~\ref{M(L_t) < M(cL): Prop: Lr-Lq ODE for m_s(tau)} imply a pointwise domination of averages through the Hardy--Littlewood maximal operator $\cM$ by splitting $\R^n$ into suitable dyadic annuli. For an explicit statement of the following corollary, we refer to \cite[Lem.~5.3]{Bechtel-Ouhabaz_ODEs}. For clarity, we write
    \begin{align*}
        (\Avg_{q,s}f)(x) \coloneqq \bigg(\fint_{B(x,s)} |f|^q \, \d y\bigg)^{\frac{1}{q}}
    \end{align*}
    for $\L^q$-averages on balls.
    
	\begin{lemma} \label{M(L_t) < M(cL): Lem: Avg dom by M}
		There is $C > 0$ such that for all $s > 0$, $\tau \in \R^*$ and $f \in \L^r \cap \L^2$ we have 
		\begin{equation*}
			\Avg_{q,s} \big(|m_{s, N}(\tau) f|^q + |\tau m_{s, N}'(\tau) f|^q\big)\leq C \cM (|f|^r)^{\frac{1}{r}}, 
		\end{equation*}
        everywhere on $\R^n$.
	\end{lemma}
	
        The domination through the maximal function implies square function estimates by a line of reasoning that goes back to Kunstmann--Weis~\cite{Kunstmann-Weis_book} and that we learned from \cite[Prop.~5.8]{Bechtel-Ouhabaz_ODEs}. We use it to provide the missing piece for the proof of Proposition~\ref{M(L_t) < M(cL): Prop: Goal}.

	\begin{proof}[\emph{Proof of Proposition~\ref{M(L_t) < M(cL): Prop: SFE for m_s(tau)}}]
		We choose $N \in \N$ as in Proposition~\ref{M(L_t) < M(cL): Prop: Lr-Lq ODE for m_s(tau)}. Given finitely many $f_j \in \L^p \cap \L^2$ and $\tau_j \in \R^*$, we first use the lower square function estimates for the family $(\Avg_{q,s})_{s>0}$ from \cite[Prop.~8.13]{Kunstmann-Weis_book} in order to write 
		\begin{align*}
			\bigg\| \Big( \sum_j |(m_{s, N}(\tau_j) + \tau_j m_{s, N}'(\tau_j)) f_j|^2  \Big)^{\frac{1}{2}} \bigg\|_p &\les \bigg\| \Big( \sum_j \big| \Avg_{q,s}\big((m_{s, N}(\tau_j) + \tau m_{s, N}'(\tau)) f_j\big)\big|^2  \Big)^{\frac{1}{2}} \bigg\|_p. 
			\intertext{Lemma~\ref{M(L_t) < M(cL): Lem: Avg dom by M} controls the right-hand side by }
			&\les \bigg\| \Big( \sum_j | \cM (|f_j|^r) |^{\frac{2}{r}}  \Big)^{\frac{1}{2}} \bigg\|_p
            \\&= \bigg\| \Big( \sum_j | \cM (|f_j|^r) |^{\frac{2}{r}}  \Big)^{\frac{r}{2}} \bigg\|_{\frac{p}{r}}^{\frac{1}{r}}
			\intertext{with exponents $\nicefrac{2}{r}, \nicefrac{p}{r} > 1$ and the classical Fefferman--Stein inequality \cite[Thm.~1]{Fefferman-Stein_Max_inequalities} yields} 
			&\les \bigg\| \Big( \sum_j | f_j |^2  \Big)^{\frac{r}{2}} \bigg\|_{\frac{p}{r}}^{\frac{1}{r}}
   \\&=\bigg\| \Big( \sum_j | f_j |^2  \Big)^{\frac{1}{2}} \bigg\|_p. &\qedhere 
		\end{align*}
	\end{proof}

    
\section{Proof of Theorem~\ref{Meyers meets DB: Main Thm: Easy formulation} in dimension $n \geq 2$ and a generalization} \label{Sec: Proof main n > 2}

According to Proposition~\ref{M(cL) < M(L_t): Prop: M(cL) < M(L_t)} and \ref{M(L_t) < M(cL): Prop: Goal}, we have $\Meyers(\cL) = \Meyers((L_\tau)_{\tau \in \R^*})$ and according to Propositions~\ref{M(L_t) = P_+(L_t): Prop: M(L_t) = P_+(L_t)} and \ref{DB: Prop: I(DB) = Hodge((L_tau)_tau)} this number coincides with $p_+(DB)$. This concludes the proof of Theorem~\ref{Meyers meets DB: Main Thm: Easy formulation}.

Our argument also gives a description of the full interval $I(DB)$ from Definition~\ref{DB: Def: Auscher-Stahlhut interval}, not just its upper endpoint. To state the general result, we write $\cL^*$ for the adjoint of $\cL$ and, with a slight abuse of notation, we call the corresponding operators in one dimension lower $L^*_\tau$ so that $L_{-\tau}^* = (L_\tau)^*$. 

\begin{definition}
   We introduce the following quantities and sets:
     \begin{align*}
        q_+(\cL) 
        &\coloneqq \sup \Big\{ p \geq 2 : (s \nabla_{t, x} (1 + s^2 \cL)^{-1})_{s > 0} \text{ is $\L^p$-bounded} \Big\},\\
        \cP(\cL) 
        &\coloneqq \Big\{ p \in (1, \infty) : \nabla_{t, x} \cL^{-1} \Div_{t, x} \text{ is $\L^p$-bounded} \Big\},\\
        \cP((L_\tau)_{\tau \in \R^*}) 
        &\coloneqq \Big\{ p \in (1, \infty) : (S_\tau L_\tau^{-1} S_\tau^*)_{\tau \in \R^*} \text{ is $\L^p$-bounded} \Big\}.
    \end{align*} 
\end{definition}

\begin{theorem}[General version of Theorem~\ref{Meyers meets DB: Main Thm: Easy formulation}] \label{Thm: characterization of I(DB)}
    In dimension $n \geq2$, we have
    \begin{align*}
        I(DB) 
        &= \cP((L_\tau)_{\tau \in \R^*})
        \\&= \cP(\cL)
        \\&= (\Meyers(\cL^*)', \Meyers(\cL))  
        \\&= (\Meyers((L_\tau^*)_{\tau \in \R^*})', \Meyers((L_\tau)_{\tau \in \R^*})) 
        \\&= (q_+(\cL^*)', q_+(\cL)).
    \end{align*}
    \end{theorem}

    \begin{proof}
    The first equality is due to Proposition~\ref{DB: Prop: I(DB) = Hodge((L_tau)_tau)} and both sets are open intervals. The statements of Proposition~\ref{M(L_t) = P_+(L_t): Prop: M(L_t) = P_+(L_t)}, Corollary~\ref{M(L_t) = P_+(L_t): Cor: M(cL) = P_+(cL)}, Proposition~\ref{M(cL) < M(L_t): Prop: M(cL) < M(L_t)} and Proposition~\ref{M(L_t) < M(cL): Prop: Goal} is that the sets in the first four lines have the same upper endpoint. The equality of line two with line five is due to \cite[Thm.~13.12]{Auscher-Egert_book}. Since these results are also true for $\cL^*$ in place of $\cL$, the sets in lines two to five coincide. Finally, the duality relation $(S_\tau L_\tau^{-1} S_\tau^*)^* = S_\tau L_{-\tau}^* S_\tau^*$ implies
    \begin{align*}
        \cP((L_\tau^*)_{\tau \in \R^*}) 
        = \Big\{ p' \in (1, \infty) : p \in \cP((L_\tau)_{\tau \in \R^*}) \Big\}.
    \end{align*}
    Consequently, the lower endpoint of $\cP((L_\tau)_{\tau \in \R^*})$ is the H\"older conjugate of the upper 
    endpoint of $\cP((L^*_\tau)_{\tau \in \R^*})$ as claimed.
    \end{proof}

    \begin{remark}
    The identity $I(DB) = (q_+(\cL^*)', q_+(\cL))$ has a non-trivial consequence due to \cite[Thm.~1.2]{Bohnlein-Egert_Lp_estimates}. Namely, not only is this set open but there is some $\varepsilon > 0$ depending only on $\lambda$ and $\| A \|_\infty$ (but especially not on the dimension $n$) such that $(2 - \varepsilon, 2 + \varepsilon) \subseteq I(DB)$.   
    \end{remark}

    \begin{remark}
    In the block case $b=0$, $c=0$, it was previously shown in \cite[Prop.~15.1]{Auscher-Egert_book} that $p_+(DB) = q_+(L_0)$. Hence, Theorem~\ref{Thm: characterization of I(DB)} yields $q_+(L_0) = q_+(\cL)$. It is possible to prove this equality directly.
    \end{remark}

 \section{Proof of Theorem~\ref{Meyers meets DB: Main Thm: Easy formulation} in dimension $n = 1$} \label{Sec: M(L) = infinity if n = 2}

In this section, we include the proof of Theorem~\ref{Meyers meets DB: Main Thm: Easy formulation} in dimension $n=1$. In this case, it is known that $p_+(DB) = \infty$, see \cite[Prop.~3.11]{Auscher-Stahlhut_Diss}. Hence, we need to show $\Meyers(\cL) = \infty$. In fact, we obtain the slightly stronger result that the weak reverse H\"older bound with $p=\infty$ holds.

  \begin{proposition} \label{M(L) in 2d: Prop: M(L) = infty}
     Let $n = 1$. There is $C > 0$ that only depends on $\lambda$ and $\|A\|_\infty$, such that for all axes-parallel cubes $Q \subseteq \R^{1+n}$ and every $\cL$-harmonic $U$ in $2 Q$ we have
     \begin{equation*}
         \| \nabla_{t,x} U \|_{\L^\infty(Q)} \leq C \left( \fint_{2 Q} |\nabla_{t,x} U |^2 \, \d (t, x) \right)^{\frac{1}{2}}.
     \end{equation*}
 \end{proposition}

\begin{proof}
For any fixed $x_0 \in \R$ and $r>0$, the transformed coefficients $A(x_0 + rx)$ are of the same class as $A$ with the same ellipticity bounds. Thus, it suffices to treat the case $Q = Q(0, 1)$. We split vectors $f \in \C^{1+1}$ as $f = [f_\perp, f_\parallel]^\top$ and, as in Section~\ref{Sec: DB}, we write
\begin{equation*}
    \underline{A} = \begin{bmatrix}
        1 & 0\\
        c & d
    \end{bmatrix}, \quad \text{so that} \quad 
    \underline{A}^{-1} = \begin{bmatrix}
        1 & 0\\
        -d^{-1}c & d^{-1}
    \end{bmatrix}.
\end{equation*}
Note that $d$ is invertible with $\|d^{-1} \|_\infty \leq \lambda^{-1}$ since $A$ is elliptic. We introduce $V \coloneqq \underline{A} \nabla_{t,x}U$ and aim for the bound
\begin{align}
\label{n=1: eq1}
    \|V\|_{\L^\infty(Q)} \leq C \|\nabla_{t,x} U\|_{\L^2(2Q)}.
\end{align}
This yields the claim because $|Q|=1$ and we have the pointwise comparability $|V| \simeq |\nabla_{t,x} U|$ with implicit constants depending only on $\lambda$ and $\|A\|_\infty$. 

The function $U$ is qualitatively smooth in $t$ since $A$ is $t$-independent (see \cite[App.~B, Lem.~1]{Auscher-Tchamitchian_book}). Within $2Q$, we compute
\begin{equation*}
    \nabla_{t,x} V_\parallel = \nabla_{t,x} (A \nabla_{t,x} U)_\parallel = \begin{bmatrix}
        \partial_t (A \nabla_{t,x} U)_\parallel \\
        \partial_x (A \nabla_{t,x} U)_\parallel
    \end{bmatrix} = \begin{bmatrix}
        (A \nabla_{t,x} \partial_t U)_\parallel \\
        - (A \nabla_{t,x} \partial_t U)_\perp
    \end{bmatrix},
\end{equation*}
where we have used the equation $\cL U = 0$ in the final step. We also have $\nabla_{t,x} V_\perp = \nabla_{t,x} \partial_t U$, which altogether leads to the pointwise control $|\nabla_{t,x} V| \les |\nabla_{t,x} \partial_t U|$. Since we are working in dimension $1+n=2$, Sobolev embeddings yield for any $p>2$ a bound
\begin{align*}
    \|V\|_{\L^\infty(Q)} 
    &\les \|V\|_{\L^p(Q)} + \|\nabla_{t,x} V\|_{\L^p(Q)} 
    \\ &\les \|\nabla_{t,x} U\|_{\L^p(Q)} + \|\nabla_{t,x} \partial_t U\|_{\L^p(Q)}.
    \intertext{We pick $p>2$ such that we have the classical Meyers estimate~\cite{Meyers_Reverse-Holder} for $\cL$-harmonic functions at our disposal. Since $\partial_t U$ is $\cL$-harmonic by $t$-independence of the coefficients, Meyers estimate applies to $U$ and $\partial_t U$, allowing us to continue by}
    &\les \|\nabla_{t,x} U\|_{\L^2(\frac{3}{2}Q)} + \|\nabla_{t,x} \partial_t U\|_{\L^2(\frac{3}{2}Q)}.
\end{align*}
Now, \eqref{n=1: eq1} follows from the Caccioppoli inequality for $\partial_t U$.
 \end{proof}

 
	\section{Systems and open problems} \label{Sec: Additional findings}

    Interested readers can check that all results in this paper remain valid for elliptic systems as long as one requires a pointwise (also known as uniformly strong) ellipticity condition. 
    
    Several references including \cite{Amenta-Auscher_Thesis, Auscher-Axelsson_weighted_max_reg, A-M_Rep-and-Uniqu-via-FO,  Auscher-Stahlhut_Diss} introduce the first-order approach under a weaker ellipticity assumption that does not imply that $B$ is invertible in $\L^\infty$. In this case, $p_+(DB)$ has a slightly more complicated definition, compare with Remark~\ref{DB: Rem: B invertible} and the proof of Proposition~\ref{DB: Prop: I(DB) = Hodge((L_tau)_tau)} fails. It remains as an open question whether Theorem~\ref{Meyers meets DB: Main Thm: Easy formulation} is still true for elliptic systems under this weaker algebraic requirement. 
    
    Another open question is whether one can prove a characterization similar to Theorem~\ref{Meyers meets DB: Main Thm: Easy formulation} if one replaces $DB$ in $n$ dimensions by more general Dirac-type operators as in \cite{A-K-McI_quadratic_estimates, Frey-McIntosh-Portal}. In this case, it is even unclear, what the corresponding equation in $(1+n)$ dimensions should be.



\begin{thebibliography}{10}
		\providecommand{\url}[1]{{\tt #1}}
		\providecommand{\urlprefix}{URL}
		\providecommand{\eprint}[2][]{\url{#2}}
		
		\bibitem{Amenta-Auscher_Thesis}
		\textsc{A.~Amenta} and \textsc{P.~Auscher}.
		\newblock Elliptic boundary Value Problems with Fractional Regularity Data, CRM Monograph Series, vol.~37, 
		\newblock AMS, Providence, 2018. 
		
		\bibitem{Auscher-Axelsson_weighted_max_reg}
		\textsc{P.~Auscher} and \textsc{A.~Axelsson}.
		{\em Weighted maximal regularity estimates and solvability of non-smooth elliptic systems. {I}} Invent. Math., \emph{184} (2011), no.~1, 47--115. 
		
		
		\bibitem{A-A-McI_L2_BVP}
		\textsc{P.~Auscher}, \textsc{A.~Axelsson} and \textsc{A.~McIntosh}.
		{\em Solvability of elliptic systems with square integrable boundary data}. Ark. Mat., \emph{48} (2010), no.~2, 253--287. 

            \bibitem{ABES-Non-Local}
            \textsc{P.~Auscher}, \textsc{S.~Bortz}, \textsc{M.~Egert}, and \textsc{O.~Saari},  
            {\em Nonlocal self-improving properties: a functional analytic approach},  
            Tunis. J. Math. \emph{1} (2019), no.~2, 151--183.  
		
		\bibitem{Auscher-Egert_book}
		\textsc{P.~Auscher} and \textsc{M.~Egert}.
		\newblock Boundary value problems and Hardy spaces for elliptic systems with block structure, Progress in Mathematics, vol.~346,
		\newblock Birkhäuser, Cham, 2023.

            \bibitem{Auscher-Egert}
		\textsc{P.~Auscher} and \textsc{M.~Egert}.
		{\em On uniqueness results for {D}irichlet problems of elliptic
              systems without de {G}iorgi--{N}ash--{M}oser regularity}. Anal. PDE, \emph{13} (2020), no.~6, 1605--1632.  
		
		
		
		\bibitem{A-M_Rep-and-Uniqu-via-FO}
		\textsc{P.~Auscher} and \textsc{M.~Mourgoglou}.
		{\em Representation and uniqueness for boundary value elliptic problems via first order systems}. Rev. Mat. Iberoam., \emph{35} (2019), no.~1, 241--315.

        	\bibitem{Auscher-Stahlhut_DB-bisectoriality}
		\textsc{P.~Auscher} and \textsc{S.~Stahlhut}.
		{\em Remarks on Functional Calculus for Perturbed First-order Dirac Operators.} Oper. Theory: Adv. Appl. \emph{240} (2014), 31--43.
		
		\bibitem{Auscher-Stahlhut_Diss}
		\textsc{P.~Auscher} and \textsc{S.~Stahlhut}.
		{\em Functional calculus for first order systems of {Dirac} type and boundary value problems.} M{\'e}m. Soc. Math. Fr., Nouv. S{\'e}r. \emph{144} (2016), 1--127, 157--164.

  \bibitem{Auscher-Tchamitchian_book}
\textsc{P.~Auscher} and \textsc{P.~Tchamitchian}.
\newblock Square root problem for divergence operators and related topics, Ast{\'e}risque, vol.~249,
\newblock Soci{\'e}t{\'e} Math{\'e}matique de France, Paris, 1998.
		
		\bibitem{A-K-McI_quadratic_estimates}
		\textsc{A.~Axelsson} , \textsc{S.~Keith} and \textsc{A.~McIntosh}.
		{\em Quadratic estimates and functional calculi of perturbed {Dirac} operators}. Invent. Math., \emph{163} (2006), no.~3, 455--497.
		
		\bibitem{Bechtel_Lp}
		\textsc{S.~Bechtel}.
		{\em $\L^p$-estimates for the square root of elliptic systems with mixed boundary conditions $\mathrm{II}$}. J. Differential Equations, \emph{379} (2024), 104--124.
		
		\bibitem{Bechtel-Ouhabaz_ODEs}
		\textsc{S.~Bechtel} and \textsc{E.~M.~Ouhabaz}.
		{\em Off-diagonal bounds for the Dirichlet--to--Neumann operator on Lipschitz domains}. J. Math. Anal. Appl., \emph{530} (2024), no.~2, 18.
		
		\bibitem{Bohnlein-Egert_Lp_estimates}
		\textsc{T.~Böhnlein} and \textsc{M.~Egert}.
		{\em Explicit improvements for $\L^p$-estimates related to elliptic systems}. Bull. Lon. Math. Soc., \emph{56} (2024), no.~3, 914--930.

            \bibitem{Dahlberg}
            \textsc{B.E.J.~Dahlberg}.
            {\em Estimates of harmonic measure}. Arch. Rational Mech. Anal., \emph{65} (1977), no.~3, 275--288.
		
		
		\bibitem{Fefferman-Stein_Max_inequalities}
		\textsc{C.~Fefferman} and \textsc{E.~M.~Stein}.
		{\em Some maximal inequalities}. Am. J. Math., \emph{93} (1971), no.~1, 107--115.

        \bibitem{Frey-McIntosh-Portal}
\textsc{D.~Frey}, \textsc{A.~McIntosh} and \textsc{P.~Portal}.
{\em Conical square function estimates and functional calculi for perturbed {Hodge}-{Dirac} operators in {{\(L^p\)}}}. J. Anal. Math. \emph{134} (2018), no.~2, 399--453.

        \bibitem{Gehring_Quasiconformal_mapping}
        \textsc{F.~W.~Gehring}.
        {\em The $\L^p$-integrability of the partial derivatives of a quasiconformal mapping}. Acta Math., \emph{130} (1973), 265--277.
		

        \bibitem{Grafakos-Modern_book}
		\textsc{L.~Grafakos}.
		\newblock Modern Fourier Analysis, ed.~3rd, Graduate Texts in Mathematics, vol.~250,
		\newblock Springer, New York, 2014.

            \bibitem{HKMP}
		\textsc{S.~Hofmann}, \textsc{C.~Kenig}, \textsc{S.~Mayboroda} and \textsc{J.~Pipher}.
		{\em Square function/non-tangential maximal function estimates and
              the {D}irichlet problem for non-symmetric elliptic operators}. J. Amer. Math. Soc., \emph{28} (2015), no.~2, 483--529.

            \bibitem{HKMP2}
		\textsc{S.~Hofmann}, \textsc{C.~Kenig}, \textsc{S.~Mayboroda} and \textsc{J.~Pipher}.
		{\em The regularity problem for second order elliptic operators
              with complex-valued bounded measurable coefficients}. Math. Ann., \emph{361} (2015), no.~3-4, 863--907.

        \bibitem{Analysis_in_BS_II}
		\textsc{T.~Hyt\"onen}, \textsc{J.~van Neerven}, \textsc{M.~Veraar} and \textsc{L.~Weis}.
		\newblock Analysis in Banach Spaces, Volume II: Probabilistic Methods and Operator Theory, Ergebnisse der Mathematik und ihrer Grenzgebiete. 3. Folge, vol.~67,
		\newblock A Series of Modern Surveys in Mathematics, Springer, Cham, 2017. 
        
		
		\bibitem{Kato_book}
		\textsc{T.~Kato}.
		\newblock Perturbation Theory for Linear Operators, Classics in Mathematics, vol.~132,
		\newblock Springer, Berlin, 1995.
		
		\bibitem{Kunstmann-Weis_book}
		\textsc{P.~C.~Kunstmann} and \textsc{L.~Weis}.
		\newblock Maximal $L_p$-regularity for parabolic equations, Fourier multiplier theorems and $H^{\infty}$-functional calculus. 
		\newblock Lecture Notes in Mathematics 1855, Springer, Berlin, 2004, 65--311. 

             \bibitem{Meyers_Reverse-Holder}
            \textsc{N.~Meyers}.
            {\em An $L^p$-estimate for the gradient of solutions of second order elliptic divergence equations}. Ann. Sc. Norm. Super. Pisa, Sci. Fis. Mat., III., \emph{17} (1963), no.~3, 189--206.

            \bibitem{Rosen}
            \textsc{A.~Ros\'{e}n}.
            {\em Layer potentials beyond singular integral operators.} Publ. Mat., \emph{57} (2013), no.~2, 429--454.
		
		\bibitem{Shen_Lp-extrapolation}
		\textsc{Z.~Shen}.
		{\em Bounds of Riesz Transform on $L^p$ Spaces for Second Order Elliptic Operators}. Ann. Inst. Fourier., \emph{55} (2005), no.~1, 173--197.	
		
		\bibitem{Sneiberg_Extrapolation}
		\textsc{I.~\u{S}ne\u{\ii}berg}.
		{\em Spectral properties of linear operators in interpolation families of Banach spaces}. Mat. Issled., \emph{9} (1974), no.~2, 214--229, 254--255.	

        \bibitem{Tolksdorf_R-sectoriality}
\textsc{P.~Tolksdorf}. {\em $\mathcal{R}$-sectoriality of higher-order elliptic systems on general bounded domains}.
J.~Evol.~Eq. \emph(2018), no.~2, 323--349.
		
		\bibitem{Weis_Lp-multipliers}
		\textsc{L.~Weis}.
		{\em Operator-valued Fourier multiplier theorems and maximal $L_p$-regularity}. Math. Ann. \emph{319} (2001), no.~4, 735--758.
		
	\end{thebibliography}
\end{document}